\documentclass{amsart}
\usepackage{amsmath}
\usepackage{amssymb}
\usepackage{amsthm}
\usepackage{cite}
\usepackage{mathtools}
\usepackage{embrac}
\usepackage{bbm}
\usepackage[all]{xy}
\usepackage{extarrows}
\usepackage{graphicx}
\usepackage{multirow}
\usepackage{float}
\usepackage{caption}
\usepackage{subcaption}
\usepackage{enumitem}
\usepackage{hyperref}

\newtheorem{theorem}{Theorem}[section]

\newtheorem{remark}[theorem]{Remark}
\newtheorem{definition}[theorem]{Definition}
\newtheorem{proposition}[theorem]{Proposition}
\newtheorem{lemma}[theorem]{Lemma}

\def\ZZ{{\mathbb Z}}

\def\QQ{{\mathbb Q}}

\def\FF{{\mathbb F}}

\def\SS{{\mathcal S}}

\def\O{{\mathcal O}}
\def\1{{\mathbbm 1}}

\newcommand{\rr}{\rightarrow}

\newcommand{\op}{\operatorname}

\newcommand{\quash}[1]{}

\begin{document}

\title{On the quaternionic Serre Weights}
\author{Yang CHEN \and Haoran WANG}
\date{}

\maketitle
\begin{abstract}
    We give a representation theoretic description of the set of quaternionic Serre weights for generic two-dimensional mod $p$ representations of $\operatorname{Gal}(\overline{\QQ}_p / K)$, where $K$ is a finite unramified extension of $\QQ_p$.    
\end{abstract}
\tableofcontents
\section{Introduction}

Let $p$ be a prime number. Serre \cite{MR885783} conjectured that every odd irreducible continuous representation $
\overline{r}: {\rm Gal}(\overline{\mathbb{Q}}/\mathbb{Q}) \to \mathrm{GL}_{2}\left(\overline{\mathbb{F}}_{p}\right)
$ is modular in the sense that it arises from some modular eigenform. Serre also formulated a refined conjecture which predicts the minimal 
weight and level of the modular eigenforms. The original Serre's conjecture has been proved by Khare and Wintenberger \cite{Khare-Wintenberger}. In fact Serre \cite{MR885783} asked whether a ``mod $p$ Langlands philosophy'' exists, which is now known to be true for ${\rm GL}_2$ over $\mathbb{Q}$ \cite{Emerton} \cite{BreuilICM}. 

In \cite{MR2730374} Buzzard, Diamond and Jarvis generalized Serre's refined conjecture to Hilbert modular forms over totally real field $F$ in which $p$ is unramified. As there is no obvious notion of minimal weight, \cite{MR2730374} used irreducible $\overline{\mathbb{F}}_{p}$-representations of $\op{GL}_{2}(\mathcal{O}_F/ p) $ as a generalization of the notion of weight, and predicted all the possible weights for totally odd irreducible continuous modular representation $\overline{r}: {\rm Gal}(\overline{F}/F) \to \mathrm{GL}_{2}\left(\overline{\mathbb{F}}_{p}\right).$ Let's call the set of all possible weights the set of ${\rm GL}_2$-Serre weights for $\overline{r}.$ The weight part of Serre's conjecture formulated in \cite{MR2730374} and its natural generalizations to definite or indefinite quaternion algebra $B$ over $F$ which splits at all places above $p$ are proved in \cite{GLS}. There are important generalizations of weight part of Serre's conjecture formulated in \cite{Schein} \cite{Herzig}  \cite{GHS}. We refer to the introduction of \cite{GHS} for a historical account.

In this paper we consider the weight part of Serre's conjecture for $\overline{r}: {\rm Gal}(\overline{F}/F) \to \mathrm{GL}_{2}\left(\overline{\mathbb{F}}_{p}\right)$ when $B$ is a definite or indefinite quaternion algebra over totally real field $F$ ramified at all places above $p.$ The possible weights are irreducible $\overline{\mathbb{F}}_{p}$-representations of $(\O_B \otimes_{\ZZ} \ZZ_p)^{\times},$ where $\O_B$ is a fixed maximal order. Let's call the set of all possible weights the set of quaternionic Serre weights for $\overline{r}.$ Recently, Scholze \cite{Scholze} proposed a mod $p$ Jacquet-Langlands correspondence which satisfies some local-global compatibility. Recent works like \cite{Ludwig} \cite{Paskunas} \cite{Hu-Wang-JL1} give some study on this mod $p$ Jacquet-Langlands correspondence. We would like to view the question of determining the set of quaternionic Serre weights as part of the mod $p$ Langlands program. Up to multiplicity, it gives the $(\O_B \otimes_{\ZZ} \ZZ_p)^{\times}$-socle of some interesting mod $p$ representation of $(B\otimes_{\mathbb{Q}}\mathbb{A}_f)^{\times}.$

When $F=\QQ,$ the set of quaternionic Serre weights is completely determined in \cite{MR1802794}. Under some Taylor-Wiles type assumption, \cite{Gee-Savitt} determined, in terms of irreducible $\overline{\mathbb{F}}_{p}$-representations of $(\O_B \otimes_{\ZZ} \ZZ_p)^{\times},$ the set of quaternionic Serre weights in most cases if $\overline{r}|_{{\rm Gal}(\overline{F}_v/F_v)}$ is {\em semisimple} for all $v|p.$ In the non-semisimple case, they gave an explicit description in terms of $p$-adic Hodge theory. However, in order to understand the mod $p$ Jacquet-Langlands correspondence, one still wants to have an explicit description in terms of representations of $(\O_B \otimes_{\ZZ} \ZZ_p)^{\times}.$ Note that \cite{Gee-Savitt} had no assumption on the ramification of $p$ in the totally real field $F.$

In this paper, we assume $p$ is unramified in $F$ and we give a representation theoretic description of the set of quaternionic Serre weights for {\em generic} 2-dimensional modular $\overline{r}$ under the Taylor-Wiles type assumption. The generic condition is the one introduced in \cite{Breuil-Paskunas} for $\overline{r}|_{{\rm Gal}(\overline{F}_v/F_v)}$ for all $v|p.$ It follows from  \cite{Gee-Savitt} that every quaternionic Serre weight is of the form $\otimes_{v|p} \sigma_v$, for $\sigma_v$ lies in a set $W_{B_v} (\overline{r}|_{{\rm Gal}(\overline{F}_v/F_v)})$ which depends only on $\overline{r}|_{{\rm Gal}(\overline{F}_v/F_v)}.$ The question is then reduced to determine the set  $W_{B_v} (\overline{r}|_{{\rm Gal}(\overline{F}_v/F_v)})$ for each $v|p.$ We let $K:= F_v$ which is unramified of degree $f$ over $\QQ_p$  with residue field $k.$ Let $\overline{\rho} : = \overline{r}|_{{\rm Gal}(\overline{F}_v/F_v)}(1),$ $D:= B_v$ and $W_D(\overline{\rho}) : = W_{B_v} (\overline{r}|_{{\rm Gal}(\overline{F}_v/F_v)}).$ We deduce our results from the existing knowledge of ${\rm GL}_2$-Serre weights (\cite{MR2392355} \cite{Breuil-Paskunas} \cite{MR3274546}). Under the generic assumption, it is easy to see that an irreducible mod $p$ representation $\overline{\psi}$ of $\mathcal{O}_D^{\times}$ is in $W_D(\overline{\rho})$ if and only if $W_{\op{GL}_2}(\overline{\rho})$ intersects with the set of Jordan-H\"older factors of the reduction mod $p$ of the tame cuspidal type $\Theta([\overline{\psi}]),$ where $[\overline{\psi}]$ is the Techm\"uller lift of $\overline{\psi},$ see Proposition \ref{prop::serre wts}\footnote{This has already been indicated in Remark 5.10 of \cite{Gee-Geraghty}} for a precise statement. This gives a way to compute $\overline{\psi}.$ We then prove that such $\overline{\psi}$'s can be  parameterised by $(\vec{w},\vec{d})$ with $\vec{w}=(w_0,...,w_{f-1})\in\{0,1\}^{\ZZ/f\ZZ}$ and $\vec{d}=(d_0,...,d_{f-1})\in\{-1,0,1\}^{\ZZ/f\ZZ}$ subject to some relations. Our main results in the semisimple case are given in Theorems \ref{26}, \ref{thm::irreducible case} respectively. We use Theorem \ref{26} as an example to illustrate our explicit description.

\begin{theorem}[Theorem \ref{26}]
Let $q = p^f$ and let $l $ be the quadratic extension of $k.$ We fix an embedding $\overline{\iota}: l \hookrightarrow \mathbb{F}$, where $\mathbb{F}$ is a sufficiently large finite extension of $\mathbb{F}_p$. Let $\overline{\kappa}_{0}: k \hookrightarrow \mathbb{F}$ be given by $\overline{\iota}^{q+1}=\overline{\kappa}_{0}\circ\op{Nm}_{l/k}$.
    Let $\overline{\rho}: \operatorname{Gal}\left(\overline{\QQ}_p / K\right) \rightarrow \mathrm{GL}_{2}\left(\mathbb{F} \right)$ be a continuous representation such that its restriction to inertia is: 
$$
\left(\begin{array}{cc}
\omega_{f}^{r_{0}+1+p\left(r_{1}+1\right)+\cdots+p^{f-1}\left(r_{f-1}+1\right)} & 0 \\
0 & 1
\end{array}\right) \otimes (\overline{\kappa}_0\circ\op{res}\circ \op{Art}_K^{-1})
$$
with $-1 \leq r_{i} \leq p-2$. Assume $\overline{\rho}$ is generic in the sense that $0\leq r_i\leq p-3$ for all $i$ and not all $r_i$
equal to $0$ or equal to $p-3$. Then $W_D(\overline{\rho})$ consists of $\overline{\psi}=\overline{\iota}^{\sum_{i=0}^{f-1} q^{w_{i}} p^{i} r_{i}+(1-q)\sum_{i=0}^{f-1}d_ip^i},$ where $w_{i} \in\{0,1\}$ and $d_{i} \in\{-1,0,1\}$ satisfying: 
    \begin{itemize}
        \item For $i>0$, 
        $w_i=\left\{\begin{array}{cl}
        1&  \text { if } d_i=-1\\
        0&  \text { if } d_i=1             
        \end{array}\right.$, 
    and if $d_i=0$, then $(w_{i-1},w_i)=(0,0)$ or $(1,1)$.
        \item 
        $w_0=\left\{\begin{array}{cl}
        1&  \text { if } d_0=-1\\
        0&  \text { if } d_0=1             
        \end{array}\right.$, 
    and if $d_0=0$, then $(w_{f-1},w_0)=(0,1)$ or $(1,0)$. 
    \end{itemize}
We write $\overline{\psi}_{\vec{w},\vec{d}}$ for such $\overline{\psi}$ and we have $\overline{\psi}_{\vec{w},\vec{d}}=\overline{\psi}_{\vec{w}',\vec{d}'}$ if and only if $(\vec{w},\vec{d})=(\vec{w}',\vec{d}').$ 
\end{theorem}

To determine the set $W_D(\overline{\rho})$ for non-semisimple $\overline{\rho},$ we first give a partition of $W_D(\overline{\rho}^{{\rm ss}})$ where $\overline{\rho}^{{\rm ss}}$ denotes the semisimplification of $\overline{\rho}.$ We identify the set of ${\rm GL}_2$-Serre weights for  $\overline{\rho}^{{\rm ss}}$ with the set of $f$-tuples $\vec{v} \in \{0,1\}^{\mathbb{Z}/f\mathbb{Z}}.$ We have 
\[
W_D(\overline{\rho}^{{\rm ss}}) = \bigsqcup\limits_{{\vec{v}} \in \{0,1\}^{\mathbb{Z}/f\mathbb{Z}}} W_{D}^{\vec{v}}(\overline{\rho}^{{\rm ss}}).
\]
Each subset $W_{D}^{\vec{v}}(\overline{\rho}^{{\rm ss}})$ can be explicited described. 

\begin{theorem}[Theorem \ref{thm::nonsplit}]
    Let $\overline{\rho}$ be the generic reducible nonsplit Galois representation as in Definition \ref{2.6}. Then
$$W_{D}(\overline{\rho})=\bigsqcup_{{\vec{v}}\leq\vec{v}(\overline{\rho})}W_{D}^{{\vec{v}}}(\overline{\rho}^{{\rm ss}}),$$
where $\leq$ is the partial order which is translated from the partial order defined in \cite[\S 11]{Breuil-Paskunas}  and $\vec{v}(\overline{\rho})\in\{0,\ldots,f-1\}^{\ZZ/f\ZZ}$ is the $f$-tuple associated to $\overline{\rho}$ which is defined in Definition \ref{2.6}.
\end{theorem}


\subsection{Acknowledgements}

We thank Prof. Yongquan Hu for several interesting discussions during the preparation of the paper. We thank Prof. Christophe Breuil and Prof. Florian Herzig for their comments on an earlier draft. The work is supported by National Natural Science Foundation of China Grants 11971028.

\section{Preliminaries}\label{s2}

We keep the notations as in the introduction. Let $K$ be unramified of degree $f$ over $\mathbb{Q}_p$ with residue field $k.$ Let $l$ be the quadratic extension of $k$ in $\overline{k}.$ Let $E$ be a finite extension of $\mathbb{Q}_p$ with integer ring $\O$ uniformizer $\varpi_E$ and residue field $\mathbb{F}.$ We assume $\mathbb{F}$ is sufficiently large. We fix an embedding $\overline{\iota}: l \hookrightarrow \mathbb{F}$. Let $\overline{\kappa}_{0}: k \hookrightarrow \mathbb{F}$ be given by $\overline{\iota}^{p^f+1}=\overline{\kappa}_{0}\circ\op{Nm}_{l/k}$. We define embeddings $\overline{\kappa}_{i}: k \hookrightarrow \mathbb{F}$ 
by $\overline{\kappa}_{i+1}^{p}=\overline{\kappa}_{i}$ for $i \in \ZZ$.

\subsection{Irreducible $\FF$-representations of $\op{GL_2}(k)$}

Irreducible $\FF$-representations of $\mathrm{GL}_2(k)$ are of the form: 
\begin{equation}\label{1}
    \overline{\sigma}_{\vec{t}, \vec{s}}:=\otimes_{j=0}^{f-1}\left(\operatorname{det}^{t_{j}} \operatorname{Sym}^{s_{j}} k^{2}\right) \otimes_{k, \overline{\kappa}_{-j}} \mathbb{F},   
\end{equation}
where $0\leq s_j,$ $t_j\leq p-1$ and not all $t_j$ are equal to $p-1$. An irreducible $\FF$-representations of $\mathrm{GL}_2(k)$ can also be written in the form: 
\begin{equation}\label{8}
    \operatorname{Sym}^{r_{0}} \FF^{2} \otimes_{\FF}\left(\operatorname{Sym}^{r_{1}} \FF^{2}\right)^{\mathrm{Fr}} \otimes \cdots \otimes_{\FF}\left(\mathrm{Sym}^{r_{f-1}} \FF^{2}\right)^{\mathrm{Fr}^{f-1}} \otimes_{\FF} (\overline{\eta} \circ \op{det})
\end{equation}
where the $r_i$ are integers between $0$ and $p-1$, $\overline{\eta}$ is a smooth character $k^{\times} \rightarrow \FF^{\times}$, $\mathrm{GL}_2(k)$ acts on the first $\operatorname{Sym}$ via the fixed embedding $\overline{\kappa}_0:k \hookrightarrow \FF$ and on the others via twists by powers of the Frobenius $\mathrm{Fr}$ where $\operatorname{Fr}(x):=x^{p}\left(x \in \mathbb{F}_{p^{f}}\right)$. 
Such representation is denoted by $\left(r_{0}, \ldots, r_{f-1}\right) \otimes (\overline{\eta} \circ \op{det})$. The relation between the above two descriptions of irreducible $\FF$-representations of $\mathrm{GL}_2(k)$ is given by the following lemma. 

\begin{lemma}\label{2.1}
   We have   $\overline{\sigma}_{\vec{t}, \vec{s}}=(s_0,...,s_{f-1})\otimes (\overline{\kappa}_0\circ \op{det}^{\sum_{i=0}^{f-1}p^it_{i}}).$      
\end{lemma} 
\begin{proof}
    This follows from $\op{Sym}^{s_j}k^{2}\otimes_{k, \overline{\kappa}_{-j}} \mathbb{F}=(\op{Sym}^{s_j}\FF^{2})^{\mathrm{Fr}^j}$ and $\op{det}^{t_j}\otimes_{k, \overline{\kappa}_{-j}} \mathbb{F}=\overline{\kappa}_0\circ \op{det}^{p^jt_j}$ for $0\leq j\leq f-1$.
\end{proof}
\subsection{Cuspidal types}
Let $\psi: l^{\times} \rightarrow \mathcal{O}^{\times}$ be a multiplicative character which does not factor through the norm $\op{Nm}_{l/k}:l^{\times}\rr k^{\times}$. Let $\Theta(\psi)$ be the cuspidal type of irreducible $E$-representations of $\op{GL}_2(k)$ as in \cite[\S 1]{MR2392355}.  Let $\overline{\Theta}(\psi)^{{\rm ss}}$ be the semisimplification of the reduction mod $\varpi_E$ of any $\op{GL}_2(k)$-stable $\O$-lattice in $\Theta(\psi).$ Then the Jordan-H\"older factors of $\overline{\Theta}(\psi)^{{\rm ss}}$ are described in \cite[Prop. 1.3]{MR2392355}. We follow the presentation of \cite[\S 3.3]{MR3323575}.

We can write $\psi$ in the form $[\overline{\iota}]^{(q+1) b+1+c}$, where $0 \leq b \leq q-2$, $0 \leq c \leq q-1$ and $[\overline{\iota}]$ is the Teichm\"uller lift of $\overline{\iota}$. Write
$c=\sum_{i=0}^{f-1} c_{i} p^{i}$, where $0 \leq c_{i} \leq p-1$. If $J \subseteq \mathcal{S}:=\{0,...,f-1\}$, we set $J_{0}=J \triangle\{f-1\}$, i.e. $J_0=J\cup\{f-1\}$ if $f-1\notin J$, and $J_0=J\backslash\{f-1\}$ if $f-1\in J$.
We define $\mathcal{P}_{\Theta(\psi)}$ to be the collection of subsets of $\SS$ consisting of those $J$ satisfying the conditions: 
\begin{itemize}
    \item if $j\in J$ and $j-1\notin J_0$ then $c_j\neq p-1$, and
    \item if $j\notin J$ and $j-1\in J_0$ then $c_j\neq 0$.
\end{itemize}
For any $J\in\mathcal{P}_{\Theta(\psi)}$ we define $s_{J, i}$ and $t_{J, i}$ by 
\begin{equation}\label{2}
    s_{J, i}=\left\{\begin{array}{cl}p-1-c_{i}-\delta_{\left(J_{0}\right)^{c}}(i-1) & \text { if } i \in J \\ c_{i}-\delta_{J_{0}}(i-1) & \text { if } i \notin J\end{array}\right., 
\end{equation}
\begin{equation}\label{3}
    t_{J, i}=\left\{\begin{array}{cl}c_{i}+\delta_{J^{c}}(i-1) & \text { if } i \in J\\
     0 & \text { if } i \notin J\end{array}\right.,
\end{equation}
where, for $J\subseteq \SS$, $\delta_J(i)=1$ if $i\in J$, and 0 otherwise.

\begin{lemma}\label{6}
The Jordan-H\"older factors of $\overline{\Theta}(\psi)^{{\rm ss}}$ are parameterised by $\mathcal{P}_{\Theta(\psi)}$ as follows: for $J\in\mathcal{P}_{\Theta(\psi)}$, let
\begin{equation}\label{4}
    \overline{\Theta}(\psi)_{J}:=\overline{\sigma}_{\vec{t}_J, \vec{s}_J} \otimes(\overline{\kappa}_{0}\circ \mathrm{det}^{b+\delta_{J}(0) \delta_{J}(f-1)+\delta_{J^c}(0) \delta_{J^c}(f-1)}).\footnote{We believe that the coefficient of $b$ in (\ref{4}) should be $1$ instead of $q+1$ in \S 3.3 of \cite{MR3323575}.} 
\end{equation}
\end{lemma}
\begin{proof}
See \cite[Prop. 1.3]{MR2392355} or \cite[\S 3.3]{MR3323575}.
\end{proof}

A subset $J$ of $\SS$ can be identified with an $f$-tuple $\vec{u}:=(u_0,...,u_{f-1})\in\{0,1\}^{\ZZ/f\ZZ}$, where $u_i=0$ if $i\notin J$ and $u_i=1$ if $i\in J$. For $\vec{u}:=(u_0,...,u_{f-1})\in\{0,1\}^{\ZZ/f\ZZ}$, define 
    $\vec{u}_0=(u_{0,0},...,u_{0,f-1})$ by letting $u_{0,i}=u_i$ for $0\leq i<f-1$ and $u_{0,f-1}=1-u_{f-1}$. 
    If $\vec{u}$ corresponds to $J$, then $\vec{u}_0$ corresponds to $J_0$ under the above identification. If $\vec{u}\in \{0,1\}^{\ZZ/f\ZZ}$ corresponds to $J \in \mathcal{P}_{\Theta(\psi)},$ then $\vec{u}$ satisfies
       \begin{itemize}
        \item if $u_i=1$ and $u_{0,j-1}=0$ then $c_j\neq p-1$, and
        \item if $u_i=0$ and $u_{0,j-1}=1$ then $c_j\neq 0$. 
    \end{itemize}
We abuse notation by letting $\mathcal{P}_{\Theta(\psi)}$ denote the set of $\vec{u}$'s satisfying above conditions.

\begin{definition}
    For $\vec{u}:=(u_0,...,u_{f-1})\in\{0,1\}^{\ZZ/f\ZZ}$, we define: 
    \begin{equation}\label{2.0}
        s_{\vec{u}, i}=\left\{\begin{array}{cl}p-2+u_{0,i-1}-c_{i} & \text { if } u_i=1 \\ c_{i}-u_{0,i-1} & \text { if } u_i=0\end{array}\right., 
    \end{equation}
    \begin{equation}\label{3.0}
        t_{\vec{u}, i}=\left\{\begin{array}{cl}c_{i}+(1-u_{i-1}) & \text { if } u_i=1\\
         0 & \text { if } u_i=0\end{array}\right., 
    \end{equation}
    \begin{equation}\label{4.0}
        \overline{\Theta}(\psi)_{\vec{u}}:=\overline{\sigma}_{\vec{t}_{\vec{u}}, \vec{s}_{\vec{u}}} \otimes(\overline{\kappa}_{0}\circ \mathrm{det}^{b+u_0u_{f-1}+(1-u_0)(1-u_{f-1})}). 
    \end{equation}
    Then if $\vec{u}\in \mathcal{P}_{\Theta(\psi)}$ corresponds to $J \in \mathcal{P}_{\Theta(\psi)}$, we have $s_{\vec{u}, i}=s_{J, i}$, $t_{\vec{u}, i}=t_{J, i}$, and $\overline{\Theta}(\psi)_{\vec{u}}=\overline{\Theta}(\psi)_{J}$. 
\end{definition}

\subsection{Automorphic forms on quaternion algebras over totally real field}
Let $F$ be a totally real field and $B$ be a quaternion algebra over $F.$ We assume $B$ is either \emph{definite} (i.e. it is ramified at all infinite places), or $B$ is \emph{indefinite} (i.e. it splits at exactly one infinite place of $F$). Let $\O_B$ be a fixed maximal order of $B.$ Since our main result is on the local side, we simplify the global setup by assuming that $p$ is {\em inert} in $F.$ Let $v$ denote the unique place of $B$ over $p.$

We define the space of automorphic forms associated to $B$ as in \cite[\S 5.2]{Hu-Wang-JL1}. Up to normalization, it is identical to the space used in \cite{Gee-Savitt} (resp. \cite{Breuil-Diamond}) when $B$ is definite (resp. indefinite). Let $\overline{r}: {\rm Gal}(\overline{F}/F) \to {\rm GL}_2(\mathbb{F})$ be an absolutely irreducible totally odd representation, which is modular in the sense of \cite[\S 3.1]{Breuil-Diamond}. Assume $\overline{r}|_{{\rm Gal}(\overline{F}/F(\sqrt[p]{1}))}$ is absolutely irreducible and, if $p = 5,$ the image of $\overline{r}({\rm Gal}(\overline{F}/F(\sqrt[p]{1})))$ in ${\rm PGL}_2(\mathbb{F})$ is not isomorphic to ${\rm PSL}_2(\mathbb{F}_5).$ Then $\overline{r}$ gives a maximal ideal of the abstract Hecke algebra which acts on the space of modular forms in the usual way, see \cite[\S 5.2]{Hu-Wang-JL1} \cite{Gee-Savitt} \cite{Breuil-Diamond}. We denote by $\pi^B(\overline{r})$ the admissible smooth representation of $(B\otimes_{\mathbb{Q}} \mathbb{A}_f)^{\times}$ over $\mathbb{F}$ given by \cite[(5.4)]{Hu-Wang-JL1}. When $B$ is indefinite, $\pi^B(\overline{r})$ is the representation denoted by $\pi_D(\overline{\rho})$ in \cite[\S 3.1]{Breuil-Diamond} up to twist. We assume the hypothesis (H0) of \cite{Breuil-Diamond} is satisfied. Then by \cite[Cor. 3.2.3]{Breuil-Diamond}, $\pi^B(\overline{r})$ is non-zero. 

Let $\sigma$ be an irreducible representation of $\O_{B_v}^{\times}$ over $\mathbb{F}.$ We say that $\sigma$ is a Serre weight for $\overline{r}$ with respect to $B$ (at $v$) if 
\begin{equation} \label{eq::def-SW}
{\rm Hom}_{\O_{B_v}^{\times}} (\sigma,\pi^B(\overline{r}) ) \neq 0.
\end{equation}
Let $W_B(\overline{r})$ denote the set of $\sigma$'s satisfying \eqref{eq::def-SW}. If $B$ splits at $v,$ \cite{GLS} proves that the set $W_B(\overline{r})$ depends only on $\overline{r}|_{G_{I_v}},$\footnote{In fact \cite{GLS} establishes this also in non-generic case.} where $I_{F_v}$ is the inertia subgroup of $G_{F_v} : = {\rm Gal}(\overline{F}_v/F_v)$. If moreover $\overline{r}|_{G_{F_v}}$ is generic in the sense of \cite{Breuil-Paskunas}, then the set $W_B(\overline{r})$ is equal to the set $W_{\rm GL_2}(\overline{\rho})$ with $\overline{\rho} : = \overline{r}|_{G_{F_v}}(1).$ The generic condition of \cite{Breuil-Paskunas} and the set  $W_{\rm GL_2}(\overline{\rho})$ are recalled in \S \ref{sec::SW for GL2} with $K=F_v$.

Assume $B$ is ramified at $v.$ Then irreducible representations of $\O_{B_v}^{\times}$ over $\mathbb{F}$ are the same as multiplicative characters $\overline{\psi}: l^{\times} \to \mathbb{F}^{\times},$ where $l$ is the quadratic extension of the residue field $k$ of $F_v.$ If moreover $\overline{r}|_{G_{F_v}}$ is generic, then \cite[Thm. 8.3]{Gee-Savitt} shows that $W_B(\overline{r}) $ depends only on $\overline{r}|_{I_{F_v}}$ and it only consists of characters of type I defined in \cite[Def. 3.1]{Gee-Savitt}. Recall $\overline{\psi}: l^{\times} \rightarrow \mathbb{F}^{\times}$ is a type I character if it does not factor through the norm $\op{Nm}_{l/k}:l^{\times}\rr k^{\times}$. In the following, we will write $D$ for the quaternion algebra $B_v$ over $F_v.$ So we denote $W_D(\overline{\rho}) : =W_B(\overline{r}) $ with $\overline{\rho} : = \overline{r}|_{G_{F_v}}(1) $ as above.

\begin{proposition}\label{prop::serre wts} 
Let $\overline{\psi}: l^{\times} \rightarrow \mathbb{F}^{\times}$ be a type I character. Then $\overline{\psi} \in W_D(\overline{\rho})$ if and only if ${\rm JH}\left(\overline{\Theta}(\psi)^{{\rm ss}}\right)\cap  W_{\mathrm{GL}_2}(\overline{\rho}) \neq \emptyset, $ where $\psi$ denotes the Techm\"uller lift of $\overline{\psi}.$
\end{proposition}

\begin{proof}
Let $B$ be a definite quaternion algebra over $F$ which ramified at $v.$ Via the natural projection map $\O_{B_v}^{\times} \to l^{\times},$ $\psi$ can be viewed as a character of $\O_{B_v}^{\times}.$ By the classical local Langlands and Jacquet-Langlands correspondence, $\psi$ is a $K$-type for the Weil-Deligne type $[\psi\oplus \psi^{q_v} , N = 0]$ in the sense of \cite[\S 3.2]{Breuil-Diamond}, where $q_v$ is the cardinality of $k$. By \cite[Lem. 3.2.1]{Breuil-Diamond}, there is an automorphic representation $\pi^B$ of $(B\otimes_{\mathbb{Q}} \mathbb{A})^{\times}$ whose weight is the trivial representation, such that $\overline{r} \cong r_{\pi^B} \mod \varpi_E$ and $r_{\pi^B}|_{G_{F_v}}$ is of inertial type $[\psi\oplus \psi^{q_v} , N=0].$ Here $ r_{\pi^B}$ is the Galois representation associated to $\pi^B.$

By global Jacquet-Langlands correspondence, there is an automorphic representation $\pi$ of ${\rm GL}_2(\mathbb{A}_F)$ with the same infinitesimal character as the trivial representation such that $\overline{r} \cong r_{\pi} \mod \varpi_E,$ where $ r_{\pi}$ is the Galois representation associated to $\pi,$ and $r_{\pi}|_{G_{F_v}}$ is of inertial type $[\psi\oplus \psi^{q_v} , N=0].$ Then $\Theta(\psi)$ is a $K$-type (in the sense of  \cite[\S 3.2]{Breuil-Diamond}) of the supercuspidal representation $\pi_v$, see the proof of \cite[Lem. 3.3]{Gee-Savitt}. By \cite[Lem. 3.2.1]{Breuil-Diamond} again, $\overline{r}$ is modular of weight $\sigma,$ with respect to ${\rm GL}_2$ at $v,$ for some $\sigma\in {\rm JH}\left(\overline{\Theta}(\psi)^{{\rm ss}}\right).$ 

For the converse, we may reverse the above argument. This finishes the proof of the proposition.
\end{proof}

\subsection{Serre weights for $\mathrm{GL}_2$}\label{sec::SW for GL2}
We recall the explicit description of $W_{\mathrm{GL}_2}(\overline{\rho})$ associated to generic $\overline{\rho}$ in \cite[\S 11]{Breuil-Paskunas}. For our latter use, we slightly generalize the definition of \cite{Breuil-Paskunas}. Let $K$ be the unramified extension of degree $f$ over $\mathbb{Q}_p.$
\subsubsection{The reducible case}
Let $\left(x_{0}, \ldots, x_{f-1}, x\right)$ be $f+1$ variables. We define a set $\mathcal{R} \mathcal{D}\left(x_{0}, \ldots, x_{f-1},x\right)$ of $f$-tuples 
$\lambda(x_{0}, \ldots, x_{f-1},x):=\left(\lambda_{0}\left(x_{0},x\right), \ldots, \lambda_{f-1}\left(x_{f-1},x\right)\right)$ where $\lambda_{i}\left(x_{i},x\right) \in \mathbb{Z}[x] \pm x_{i}$ as below. Note that our $\mathcal{R} \mathcal{D}\left(x_{0}, \ldots, x_{f-1},p\right)$ is the same as $\mathcal{R} \mathcal{D}\left(x_{0}, \ldots, x_{f-1}\right)$ defined in \cite[\S 11]{Breuil-Paskunas}. \par
If $f=1$, $\lambda_{0}\left(x_{0},x\right) \in\left\{x_{0}, x-3-x_{0}\right\}$. If $f>1$, then:
\begin{enumerate}[label=(\roman*)]
    \item $\lambda_{i}\left(x_{i},x\right) \in\left\{x_{i}, x_{i}+1, x-2-x_{i}, x-3-x_{i}\right\}$ for $i \in\{0, \ldots, f-1\}$
    \item if $\lambda_{i}\left(x_{i},x\right) \in\left\{x_{i}, x_{i}+1\right\}$, then $\lambda_{i+1}\left(x_{i+1},x\right) \in\left\{x_{i+1}, x-2-x_{i+1}\right\}$
    \item if $\lambda_{i}\left(x_{i},x\right) \in\left\{x-2-x_{i}, x-3-x_{i}\right\}$, then $\lambda_{i+1}\left(x_{i+1},x\right) \in\left\{x-3-x_{i+1}, x_{i+1}+1\right\}$
\end{enumerate}
with the conventions that $x_{f}=x_{0}$ and $\lambda_{f}\left(x_{f},x\right)=\lambda_{0}\left(x_{0},x\right).$

We use a different parameterisation of the set $\mathcal{R} \mathcal{D}\left(x_{0}, \ldots, x_{f-1},x\right)$ compared to \cite[\S 11]{Breuil-Paskunas}, see Remark \ref{2.2.1} below. 
An $f$-tuple $\lambda$ of the set $\mathcal{R} \mathcal{D}\left(x_{0}, \ldots, x_{f-1},x\right)$ can be naturally identified with an 
$f$-tuple ${\vec{v}}:=(v_0,...,v_{f-1})\in\{0,1\}^{\ZZ/f\ZZ}$, where $v_i=0$ if $\lambda_{i}\left(x_{i},x\right) \in\left\{x_i, x_{i}+1\right\}$ and $v_i=1$ 
if $\lambda_{i}\left(x_{i},x\right) \in\left\{x-2-x_{i}, x-3-x_{i}\right\}$. We set $\ell(\lambda):=\#\{0\leq i\leq f-1 \mid v_i=1\}$. 
Let $\lambda_{\vec{v}}$ denote the $f$-tuple $\lambda$ parameterised by ${\vec{v}}$.\par
For $\lambda\in\mathcal{R} \mathcal{D}\left(x_{0}, \ldots, x_{f-1},x\right)$, we define: 
\begin{equation}\label{5}
    e(\lambda):=\left\{
    \begin{array}{cl}
        \frac{1}{2}\left(\sum_{i=0}^{f-1} x^{i}\left(x_{i}-\lambda_{i}\left(x_{i},x\right)\right)\right) &\text {if }\lambda_{f-1}\left(x_{f-1},x\right) \in\left\{x_{f-1}, x_{f-1}+1\right\}\\
        \frac{1}{2}\left(x^{f}-1+\sum_{i=0}^{f-1} x^{i}\left(x_{i}-\lambda_{i}\left(x_{i},x\right)\right)\right) & \text{otherwise}
    \end{array}\right.
\end{equation}
\begin{remark}\label{2.2.1}
    As explained in \cite[\S 11]{Breuil-Paskunas}, the set $\mathcal{R} \mathcal{D}\left(x_{0}, \ldots, x_{f-1}\right)$ can be naturally identified with the set of subsets $J$ of $\{0,\ldots,f-1\}$ as follows: 
    set $i\in J$ if and only if $\lambda_i(x_i)\in\{p-3-x_i, x_i+1\}$. Since $v_i=1$ in our parameterisation if and only if $i+1\in J$, our $\ell(\lambda)(x_0,\ldots,x_{f-1},x)$ is the same as the $\ell(\lambda)(x_0,\ldots,x_{f-1},p)$ in \cite[\S 11]{Breuil-Paskunas} which is defined as the cardinality of $J$. 
\end{remark}
We first consider the case where $\overline{\rho}$ is reducible split. 
Let $\omega_f$ denote the fundamental character of inertia group $I_K$ of niveau $f$ associated to $\overline{\kappa}_{0}$; that is, $\omega_f$ is the $reciprocal$ of the character defined by the composite
$$
I_{K} \stackrel{\operatorname{Art}_{K}^{-1}}{\longrightarrow} \mathcal{O}_{K}^{\times} \stackrel{\op{res}}{\longrightarrow} k^{\times} \stackrel{\overline{\kappa}_{0}}{\longrightarrow} \mathbb{F}^{\times}.
$$
Let $\omega_{2f}$ be the fundamental character of niveau $2f$ associated to the embedding $\overline{\iota}$.
\begin{lemma}\label{7}
    Let $\overline{\rho}: \operatorname{Gal}\left(\overline{\QQ}_p / K\right) \rightarrow \mathrm{GL}_{2}\left(\FF\right)$ be a continuous representation such that its restriction to inertia is: 
$$
\left(\begin{array}{cc}
\omega_{f}^{r_{0}+1+p\left(r_{1}+1\right)+\cdots+p^{f-1}\left(r_{f-1}+1\right)} & 0 \\
0 & 1
\end{array}\right) \otimes (\overline{\eta}\circ\op{res}\circ \op{Art}_K^{-1})
$$
with $-1 \leq r_{i} \leq p-2$, where $\overline{\eta}$ is a smooth character $k^{\times} \rightarrow \FF^{\times}$. Assume $0\leq r_i\leq p-3$ for all $i$ and not all $r_i$
equal to $0$ or equal to $p-3$. Then $W_{\op{GL}_2}(\overline{\rho})$, the set of $\op{GL}_2$-Serre weights associated to $\overline{\rho}$, consists of: 
$$
\sigma_{\vec{v}}(\overline{\rho}):=\lambda_{\vec{v}}(r_0,...,r_{f-1},p)\otimes (\overline{\eta}\circ\operatorname{det}^{e(\lambda_{\vec{v}})\left(r_{0}, \ldots, r_{f-1},p\right)})
$$
for $\vec{v}\in \{0,1\}^{\ZZ/f\ZZ}$ and $\lambda_{\vec{v}}\in \mathcal{R} \mathcal{D}\left(x_{0}, \ldots, x_{f-1},x\right)$. 
\end{lemma}
\begin{proof}
    See \cite[Prop. 1.1 and Prop. 1.3]{MR2392355} and \cite[\S 3]{MR2392355}
\end{proof}
We now consider the case where $\overline{\rho}$ is reducible nonsplit. 
Let $\leq$ be the partial order on $\{0,1\}^{\ZZ/f\ZZ}$ defined by $\vec{v}\leq\vec{v}'$ if $v_i=1$, then $v_i'=1$ for every $0\leq i\leq f-1$. 
Then $\lambda_{{\vec{v}}}\leq \lambda_{{\vec{v}}'}$ in the sense of \cite[\S 11]{Breuil-Paskunas} if and only if $\vec{v}\leq\vec{v}'$. 
\begin{definition}\label{2.6}
    Let $\overline{\rho}: \operatorname{Gal}\left(\overline{\QQ}_p / K\right) \rightarrow \mathrm{GL}_{2}\left(\FF\right)$ be a continuous representation such that its restriction to inertia is:
$$
\left(\begin{array}{cc}
\omega_{f}^{r_{0}+1+p\left(r_{1}+1\right)+\cdots+p^{f-1}\left(r_{f-1}+1\right)} & \ast \\
0 & 1
\end{array}\right) \otimes (\overline{\eta}\circ\op{res}\circ \op{Art}_K^{-1})
$$
with $\ast\neq 0$ and $0 \leq r_{i} \leq p-3$ for all $i$ and not all $r_{i}$ equal to 0 or equal to $p-3$. As explained in \cite[\S 4]{MR3274546}, there exists an object $M$
of the Fontaine-Laffaille category such that $\overline{\rho} \simeq \operatorname{Hom}_{\op{Fil}^{\cdot}, \varphi .}\left(M, A_{\op {cris}} \otimes_{\mathbb{Z}_p} \mathbb{F}_p\right)$. 
Explicitly, $M$ is a free $k\otimes_{\FF_p}\FF$-module of rank $2$ which can be written as $M=M^0 \times \cdots \times M^{f-1}$ where $M^j=\FF e^j \oplus \FF f^j$, $s_j:=r_{f-j}$, 
$\operatorname{Fil}^0 M^j=M^j$, $\operatorname{Fil}^1 M^j=\operatorname{Fil}^{s_j+1} M^j=\FF f^j$, $\mathrm{Fil}^{s_j+2} M^j=0$ and 
$$
\left\{\begin{array}{ccc}
\varphi\left(e^j\right) & = & \alpha_{j+1} e^{j+1} \\
\varphi_{s_j+1}\left(f^j\right) & = & \beta_{j+1}\left(f^{j+1}+\mu_{j+1} e^{j+1}\right)
\end{array}\right.
$$
for $j \in\{0, \cdots, f-1\}, \alpha_j, \beta_j \in \FF^{\times}, \mu_j \in \FF$. Note that the nullity of $\mu_j$ is independent of the choice of $\left(e^j, f^j\right)_j$. 
We define an $f$-tuple $\vec{v}(\overline{\rho})\in\{0,1\}^{\ZZ/f\ZZ}$ associated to $\overline{\rho}$ by setting $\vec{v}(\overline{\rho})_j=1$ if and only if $\mu_{f-j}=0$. 
We also associate a non-empty subset $\mathcal{D}(\overline{\rho})\subseteq\mathcal{R} \mathcal{D}\left(x_{0}, \ldots, x_{f-1},x\right)$ to $\overline{\rho}$: 
\begin{equation}
    \mathcal{D}(\overline{\rho}):=\{\lambda_{\vec{v}}\in\mathcal{R} \mathcal{D}\left(x_{0}, \ldots, x_{f-1},x\right)\mid\vec{v}\in\{0,1\}^{\ZZ/f\ZZ},\vec{v}\leq \vec{v}(\overline{\rho})\}.
\end{equation}

\end{definition}

\begin{lemma}\label{18}
    Let $\overline{\rho}$ be generic reducible nonsplit as in Definition \ref{2.6}. Then $W_{\op{GL}_2}(\overline{\rho})$, the set of $\op{GL}_2$-Serre weights for $\overline{\rho}$, consists of: 
$$
\sigma_{\vec{v}}(\overline{\rho}):=\lambda_{\vec{v}}(r_0,...,r_{f-1},p)\otimes (\overline{\eta}\circ\operatorname{det}^{e(\lambda_{\vec{v}})\left(r_{0}, \ldots, r_{f-1},p\right)})
$$
for $\vec{v}\leq\vec{v}(\overline{\rho})$ in $\{0,1\}^{\ZZ/f\ZZ}$ and $\lambda_{\vec{v}}\in \mathcal{R} \mathcal{D}\left(x_{0}, \ldots, x_{f-1},x\right)$.  
\end{lemma}
\begin{proof}
    See \cite[Prop. A.3]{MR3274546}.
\end{proof}
\subsubsection{The irreducible case}
Let $\left(x_{0}, \ldots, x_{f-1},x\right)$ be $f+1$ variables. We define a set $\mathcal{ID}\left(x_{0}, \ldots, x_{f-1},x\right)$ of $f$-tuples $\lambda\left(x_{0}, \ldots, x_{f-1},x\right):=\left(\lambda_{0}\left(x_{0},x\right), \ldots, \lambda_{f-1}\left(x_{f-1},x\right)\right)$ where $\lambda_{i}\left(x_{i},x\right) \in \mathbb{Z}[x] \pm x_{i}$ as below. Note that our $\mathcal{I} \mathcal{D}\left(x_{0}, \ldots, x_{f-1},p\right)$ is the same as $\mathcal{I} \mathcal{D}\left(x_{0}, \ldots, x_{f-1}\right)$ defined in \cite[\S 11]{Breuil-Paskunas}. \par
If $f=1$, $\lambda_{0}\left(x_{0},x\right) \in\left\{x_{0}, x-1-x_{0}\right\}$. If $f>1$, then:
\begin{enumerate}[label=(\roman*)]
\item $\lambda_{0}\left(x_{0},x\right) \in \left\{x_{0}, x_{0}-1, x-2-x_{0}, x-1-x_{0}\right\}$ and $\lambda_{i}\left(x_{i},x\right) \in$\\$\left\{x_{i}, x_{i}+1, x-2-x_{i}, x-3-x_{i}\right\}$ if $i>0$
\item if $i>0$ and $\lambda_{i}\left(x_{i},x\right) \in\left\{x_{i}, x_{i}+1\right\}$, or $i=0$ and $\lambda_{0}\left(x_{0},x\right)\in\left\{x_{0}, x_{0}-1\right\}$, then $\lambda_{i+1}\left(x_{i+1},x\right) \in\left\{x_{i+1}, x-2-x_{i+1}\right\}$
\item if $0<i<f-1$ and $\lambda_{i}\left(x_{i},x\right) \in\left\{x-2-x_{i}, x-3-x_{i}\right\}$, or $i=0$ and $\lambda_{0}\left(x_{0},x\right) \in\left\{x-2-x_{0}, x-1-x_{0}\right\}$, then $\lambda_{i+1}\left(x_{i+1},x\right) \in$ $\left\{x_{i+1}+1, x-3-x_{i+1}\right\}$
\item if $\lambda_{f-1}\left(x_{f-1},x\right)\in\left\{x-2-x_{f-1}, x-3-x_{f-1}\right\}$, then $\lambda_{0}\left(x_{0},x\right) \in\{x-1-x_{0}, x_{0}-1\}$
\end{enumerate}
with the conventions that $x_f=x_0$ and $\lambda_{f}\left(x_{f},x\right)=\lambda_{0}\left(x_{0},x\right)$.\par
We use a different parameterisation of the set $\mathcal{I} \mathcal{D}\left(x_{0}, \ldots, x_{f-1},x\right)$ compared to \cite[\S 11]{Breuil-Paskunas}, see Remark \ref{2.3.1} below. 
An $f$-tuple $\lambda$ of the set $\mathcal{ID}\left(x_{0}, \ldots, x_{f-1},x\right)$ can be naturally identified with an 
$f$-tuple ${\vec{v}}:=(v_0,...,v_{f-1})\in\{0,1\}^{\ZZ/f\ZZ}$, where $v_i=0$ if $i>0$ and $\lambda_{i}\left(x_{i},x\right) \in\left\{x_i, x_{i}+1\right\}$ (resp. $\lambda_{0}\left(x_{0},x\right) \in\left\{x_0, x_{0}-1\right\}$), and $v_i=1$ 
if $i>0$ and $\lambda_{i}\left(x_{i},x\right) \in\left\{x-2-x_{i}, x-3-x_{i}\right\}$ (resp. $\lambda_{0}\left(x_{0},x\right) \in\left\{x-2-x_0, x-1-x_{0}\right\}$). Let $\lambda_{\vec{v}}$ denote the $f$-tuple $\lambda$ parameterised by ${\vec{v}}$.\par
For $\lambda \in \mathcal{I} \mathcal{D}\left(x_{0}, \ldots, x_{f-1},x\right)$, we define : 
\begin{equation}\label{15}
    e(\lambda):=\left\{
    \begin{array}{cl} 
        \frac{1}{2}\left(\sum_{i=0}^{f-1} x^{i}\left(x_{i}-\lambda_{i}\left(x_{i},x\right)\right)\right) &\substack{\text{if }f>1\text{ and }\lambda_{f-1}\left(x_{f-1},x\right) \in\left\{x_{f-1}, x_{f-1}+1\right\}\\\text{ (resp. $\lambda_{0}(x_{0},x)=x_{0}$)}}\\
        \frac{1}{2}\left(x^{f}-1+\sum_{i=0}^{f-1} x^{i}\left(x_{i}-\lambda_{i}\left(x_{i},x\right)\right)\right) & \text{otherwise}
    \end{array}\right.
\end{equation}
\begin{remark}\label{2.3.1}
    As explained in \cite[\S 11]{Breuil-Paskunas}, the set $\mathcal{I} \mathcal{D}\left(x_{0}, \ldots, x_{f-1}\right)$ can be naturally identified with the set of subsets $J$ of $\{0,\ldots,f-1\}$ as follows: 
    for $i>0$, set $i\in J$ if and only if $\lambda_i(x_i)\in\{p-3-x_i, x_i+1\}$; set $0\in J$ if and only if $\lambda_0(x_0)\in\{p-1-x_0, x_0-1\}$. Then $v_i=1$ if and only if $i+1\in J$. 
\end{remark}
\begin{lemma}\label{16}
    Let $\overline{\rho}: \operatorname{Gal}\left(\overline{\mathbb{Q}}_{p} / K\right) \rightarrow \mathrm{GL}_{2}\left(\FF\right)$ be a continuous representation such that its restriction to inertia is: 
$$
\left(\begin{array}{cc}
\omega_{2f}^{r_{0}+1+p\left(r_{1}+1\right)+\cdots+p^{f-1}\left(r_{f-1}+1\right)} & 0 \\
0 & \omega_{2f}^{p^f(r_{0}+1)+p^{f+1}\left(r_{1}+1\right)+\cdots+p^{2f-1}\left(r_{f-1}+1\right)}
\end{array}\right) \otimes (\overline{\eta}\circ\op{res}\circ \op{Art}_K^{-1})
$$
with $0\leq r_0\leq p-1$ and $-1\leq r_i\leq p-2$ for $i>0$. Assume $1\leq r_0\leq p-2$ and $0\leq r_i\leq p-3$ for $i>0$. 
Then $W_{\op{GL}_2}(\overline{\rho})$, the set of $\op{GL}_2$-Serre weights associated to $\overline{\rho}$, consists of: 
$$
\sigma_{\vec{v}}(\overline{\rho}):=\lambda_{\vec{v}}(r_0,...,r_{f-1},p)\otimes (\overline{\eta}\circ\operatorname{det}^{e(\lambda_{\vec{v}})\left(r_{0}, \ldots, r_{f-1},p\right)})
$$
for $\vec{v}\in \{0,1\}^{\ZZ/f\ZZ}$ and $\lambda_{\vec{v}}\in \mathcal{I} \mathcal{D}\left(x_{0}, \ldots, x_{f-1},x\right)$. 
\end{lemma}
\begin{proof}
    See \cite[Prop. 1.1 and Prop. 1.3]{MR2392355} and \cite[\S 3]{MR2392355}
\end{proof}

\section{Explicit description of the set $W_{D}(\overline{\rho})$} \label{s3}

\subsection{Preparation  for the calculation of $W_{D}(\overline{\rho})$ in semisimple case}\label{s31}
We assume $\overline\rho$ is {\em semisimple} and {\em generic} in this section. Recall that we have descriptions of $\op{JH}(\overline{\Theta}(\psi)^{{\rm ss}})$ and $W_{\mathrm{GL}_2}(\overline{\rho})$ in Lemmas \ref{6}, \ref{7}, \ref{18} and \ref{16}. 
We assume $\overline{\eta}=\overline{\kappa}_0$ in $\overline{\rho}|_{I_K}$. Our goal is to determine all the $\overline{\psi}$'s such that there exist $\vec{u}\in\mathcal{P}_{\Theta([\overline\psi])}$, $\vec{v}\in\{0,1\}^{\ZZ/f\ZZ}$ satisfying $\overline{\Theta}([\overline\psi])_{\vec{u}}=\sigma_{\vec{v}}(\overline{\rho})$. \par

Let $q=p^f.$ Let $\overline{\psi}=\overline{\iota}^{(q+1)b+1+c}$ with $0\leq b\leq q-2,$ $0\leq c\leq q-1.$ By \eqref{4.0} and Lemmas \ref{7}, \ref{16}, we have: 
$$\overline{\sigma}_{\vec{t}_{\vec{u}}, \vec{s}_{\vec{u}}} \otimes\bar{\kappa}_0 \circ \mathrm{det}^{b+u_0u_{f-1}+(1-u_0)(1-u_{f-1})}=\lambda_{\vec{v}}(r_0,\ldots,r_{f-1},p)\otimes \overline{\kappa}_0\circ\operatorname{det}^{e(\lambda_{\vec{v}})\left(r_{0}, \ldots, r_{f-1},p\right)}. $$
By Lemma \ref{2.1} we get the following equations: 
\begin{equation}\label{9}
    s_{\vec{u}, i}=\lambda_{\vec{v},i}\left(r_{i},p\right), 
\end{equation}
\begin{equation}\label{10}
    b+u_0u_{f-1}+(1-u_0)(1-u_{f-1})+\sum_{i=0}^{f-1}p^it_{\vec{u}, i}\equiv e(\lambda_{\vec{v}})(r_0,\ldots,r_{f-1},p)(\text{mod }q-1).
\end{equation}
\begin{lemma}\label{P}
    Let $\vec{u},\vec{v}\in\{0,1\}^{\ZZ/f\ZZ}$ satisfy $\overline{\Theta}([\overline\psi])_{\vec{u}}=\sigma_{\vec{v}}(\overline{\rho})$. Then $\overline\psi$ is uniquely determined by $\vec{u}, \vec{v}$, and 
    $\vec{u}\in\mathcal{P}_{\Theta([\overline\psi])}$. Write $\overline{\psi}_{\vec{u},{\vec{v}}}$ for such $\overline{\psi}$. 
\end{lemma}
\begin{proof}
    If $\overline{\psi}=\overline{\iota}^{(q+1)b+1+c}$ satisfies $\overline{\Theta}([\overline\psi])_{\vec{u}}=\sigma_{\vec{v}}(\overline{\rho})$, then $(c_i)_i$ is determined by (\ref{2.0}) and (\ref{9}), and $b$ is determined by (\ref{10}). 
    By (\ref{2.0}), (\ref{9}) and the definition of $\mathcal{P}_{\Theta(\psi)}$, it suffices to verify $\lambda_{\vec{v},i}(r_i,p)\neq -1$ which follows from the generic condition. 
\end{proof}
\begin{definition}
For $z=z(x_0,\ldots,x_{f-1},x)\in \oplus_{i=0}^{f-1}\ZZ[x]x_i\oplus \ZZ[x]\subset \ZZ[x_0,\ldots,x_{f-1},x]$, we can write $z$ uniquely as $z=L(z)+C(z)$ with 
$$L(z):=\Sigma_{i=0}^{f-1}a_i x_i, \  C(z):=\Sigma_i a'_ix^i. $$
Let $C_i(z):=a'_i$. 
Define $S(z)=S(z)(x_0,\ldots,x_{f-1},x)$ to be the unique element in $\ZZ[x_0,\ldots,x_{f-1},x]$ such that $\op{deg}_{x}(S(z))<f$ and $z\equiv S(z)\ (\op{mod}x^f-1)$. 
\end{definition}
\begin{lemma}
    We have 
    $\frac{1}{2}S(\Sigma_{i=0}^{f-1}(x^{i}(x_i- \lambda_{\vec{v},i})))\in\oplus_{i=0}^{f-1}\ZZ[x] x_i\oplus\ZZ[x]$ and
    $$e(\lambda_{\vec{v}})(r_0,...,r_{f-1},p)=\frac{1}{2}S(\Sigma_{i=0}^{f-1}(x^{i}(x_i- \lambda_{\vec{v},i})))(r_0,...,r_{f-1},p).$$   
\end{lemma}
\begin{proof}
    Since $L(x_i-\lambda_{\vec{v},i})\in\{0,2x_i\}$, we have $\frac{1}{2}L(S(\Sigma_{i=0}^{f-1}(x^{i}(x_i- \lambda_{\vec{v},i}))))\in\oplus_{i=0}^{f-1}\ZZ[x] x_i$. 
    Since $\op{deg}_x C(x_i-\lambda_{\vec{v},i})\leq 1$ and $C_1(x_i-\lambda_{\vec{v},i})\in\{0,-1\}$, we see that $\frac{1}{2}C(S(\Sigma_{i=0}^{f-1}(x^{i}(x_i- \lambda_{\vec{v},i}))))\in\ZZ[x]$ follows from the fact 
    $C_1(x_i-\lambda_{\vec{v},i})=-1$ if and only if $C_0(x_{i+1}-\lambda_{\vec{v},i+1})$ is odd. The first claim follows. 
    Since $\op{deg}_x(\Sigma_{i=0}^{f-1}(x^{i} (x_i- \lambda_{\vec{v},i})))\leq f$, 
    we obtain $S(\Sigma_{i=0}^{f-1}(x^{i} (x_i- \lambda_{\vec{v},i})))$ by substituting $x^f$ by $1$ in $x^{f-1} \lambda_{\vec{v},f-1}$ if $\op{deg}_x( \lambda_{\vec{v},f-1})=1$. This coincides with 
    the definition of $e(\lambda_{\vec{v}})$ and we get the second claim.
\end{proof}
\begin{definition}
    Motivated by (\ref{2.0}), (\ref{3.0}) and (\ref{9}), we define: 
    \begin{equation}\label{5.0}
         c_{\vec{u},\vec{v},i}(x_i,x):=\left\{\begin{array}{cl}x-2+u_{0,i-1}-\lambda_{\vec{v},i}(x_i,x)&\text{ if }u_i=1\\
        u_{0,i-1}+\lambda_{\vec{v},i}(x_i,x)&\text{ if }u_i=0
        \end{array}\right., 
    \end{equation}
    \begin{equation}\label{6.0}
         t_{\vec{u},\vec{v},i}(x_i,x):=\left\{\begin{array}{cl}c_{\vec{u},\vec{v},i}(x_i,x)+(1-u_{i-1})&\text{ if }u_i=1\\
        0&\text{ if }u_i=0
        \end{array}\right.,   
    \end{equation}
    \begin{equation}\label{b0}
        b_{\vec{u},\vec{v}}(x_0,...,x_{f-1},x):=\frac{1}{2}S(\Sigma_{i=0}^{f-1}x^{i} (x_i- \lambda_{\vec{v},i}))-S(u_0u_{f-1}+(1-u_0)(1-u_{f-1})+\Sigma_{i=0}^{f-1}x^i t_{\vec{u},\vec{v},i}). 
   \end{equation}
    We write $c_{\vec{u},\vec{v},i}$, $t_{\vec{u},\vec{v},i}$, $b_{\vec{u},\vec{v}}$ for $c_{\vec{u},\vec{v},i}(x_i,x)$, $t_{\vec{u},\vec{v},i}(x_i,x)$, $b_{\vec{u},\vec{v}}(x_0,...,x_{f-1},x)$ in the following. 
\end{definition}
\begin{lemma}   
    We have 
    \begin{equation}
        \begin{split}
            b_{\vec{u},\vec{v}}(x_0,...,x_{f-1},x)&=\frac{1}{2}(\sum_{0\leq i<f-1}x^{i} (x_i- \lambda_{\vec{v},i})+S(x^{f-1}(x_{f-1}-\lambda_{\vec{v},f-1})))\\
            &-(u_0u_{f-1}+(1-u_0)(1-u_{f-1})+\sum_{0\leq i<f-1}x^i t_{\vec{u},\vec{v},i}+S(x^{f-1}t_{\vec{u},\vec{v},f-1})),
        \end{split}
    \end{equation}
and
    \begin{equation}\label{3.6}
            1+\Sigma_{i=0}^{f-1}p^i c_i+(q+1)b\equiv (1+\Sigma_{i=0}^{f-1}x^i c_{\vec{u},\vec{v},i}+(1+x^f) b_{\vec{u},\vec{v}})(r_0,...,r_{f-1},p)~ (\op{mod }q^{2}-1).        
    \end{equation}
    where $c_{0},...,c_{f-1}$ and $b$ satisfy (\ref{2.0}), (\ref{9}) and (\ref{10}). 
\end{lemma}
\begin{proof}
    The first claim follows from $\op{deg}_x (x_i-\lambda_{\vec{v},i})\leq 1$ and $\op{deg}_x (t_{\vec{u},\vec{v},i})\leq 1$. 
    Since we compute $1+c+(q+1)b$ $\op{mod} q^2-1$, it suffices to determine $b$ $\op{mod} q-1$ and the second claim follows. 
\end{proof}

Next we calculate $L(1+\Sigma_{i=0}^{f-1}x^i c_{\vec{u},\vec{v},i}+(1+x^f) b_{\vec{u},\vec{v}})$ and $C(1+\Sigma_{i=0}^{f-1}x^i c_{\vec{u},\vec{v},i}+(1+x^f) b_{\vec{u},\vec{v}})=\Sigma_{i=0}^{2f-1}d_ix^i$ in Proposition \ref{19} and Proposition \ref{11}. \par
\begin{lemma}\label{14}
    $L( c_{\vec{u},\vec{v},i})$, $L( t_{\vec{u},\vec{v},i})$, $L( \lambda_{\vec{v},i})$ are determined by $u_i$ and $v_i$ as follows.

\begin{table}[H]
    \begin{tabular}{|ll|c|c|}
    \hline
    \multicolumn{2}{|l|}{\multirow{2}{*}{$(L( c_{\vec{u},\vec{v},i}),L( t_{\vec{u},\vec{v},i}),L( \lambda_{\vec{v},i}))$}} & \multirow{2}{*}{$u_i=0$} & \multirow{2}{*}{$u_i=1$} \\
    \multicolumn{2}{|l|}{}                    &                          &                           \\ \hline
    \multicolumn{2}{|c|}{$v_i=0$}             & \multicolumn{1}{c|}{$(x_i,0,x_i)$}   & \multicolumn{1}{c|}{$(-x_i,-x_i,x_i)$}    \\ \hline
    \multicolumn{2}{|c|}{$v_i=1$}             & \multicolumn{1}{c|}{$(-x_i,0,-x_i)$}   & \multicolumn{1}{c|}{$(x_i,x_i,-x_i)$}    \\ \hline
    \end{tabular}
\end{table}
\end{lemma}
\begin{proof}
    This follows from (\ref{5.0}), (\ref{6.0}) and the definition of $\lambda_{\vec{v},i}$. 
\end{proof}
\begin{proposition}\label{19}
    $L(1+\Sigma_{i=0}^{f-1}x^i c_{\vec{u},\vec{v},i}+(1+x^f) b_{\vec{u},\vec{v}})$ are determined by $\vec{u}$ and $\vec{v}$ as follows: 
    $$L(1+\Sigma_{i=0}^{f-1}x^i c_{\vec{u},\vec{v},i}+(1+x^f) b_{\vec{u},\vec{v}})=\Sigma_{i=0}^{f-1} x^{i+w_if}x_i,$$ where $w_i\in\{0,1\}$ and $w_i\equiv u_i+v_i\ (\op{mod}2)$.
\end{proposition}
\begin{proof}
    Since $\op{deg}_x L(x_{f-1}-\lambda_{\vec{v},f-1})\leq 1$ and $\op{deg}_x L(t_{\vec{u},\vec{v},f-1})\leq 1$, we have $$L(S(x^{f-1}(x_{f-1}-\lambda_{\vec{v},f-1})))=L(x^{f-1}(x_{f-1}-\lambda_{\vec{v},f-1})),$$ 
$$L(S(x^{f-1}t_{\vec{u},\vec{v},f-1}))=L(x^{f-1}t_{\vec{u},\vec{v},f-1}), \  \text{and}$$ 
$$L(1+\Sigma_{i=0}^{f-1}x^i c_{\vec{u},\vec{v},i}+(1+x^f) b_{\vec{u},\vec{v}})=\sum_{i=0}^{f-1} L(x^i c_{\vec{u},\vec{v},i}+(1+x^f) (-x^i t_{\vec{u},\vec{v},i}+\frac{1}{2} x^i (x_i- \lambda_{\vec{v},i}))).$$
    It follows from the table in Lemma \ref{14} that
    \begin{equation*}
        \begin{split}
            L(1+\Sigma_{i=0}^{f-1}x^i c_{\vec{u},\vec{v},i}+(1+x^f) b_{\vec{u},\vec{v}})&=\sum_{i=0}^{f-1} L(x^i c_{\vec{u},\vec{v},i}+(1+x^f) (-x^i t_{\vec{u},\vec{v},i}+\frac{1}{2} x^i (x_i- \lambda_{\vec{v},i})))\\
                         &=\sum_{\substack{0\leq i\leq f-1\\u_i=0,v_i=0}}x^ix_i+\sum_{\substack{0\leq i\leq f-1\\u_i=0,v_i=1}}x^{i+f}x_i\\ &+\sum_{\substack{0\leq i\leq f-1\\u_i=1,v_i=0}}x^{i+f}x_i
                         +\sum_{\substack{0\leq i\leq f-1\\u_i=1,v_i=1}}x^ix_i.
        \end{split}
    \end{equation*}
\end{proof}

Since $\op{deg}_x(C( c_{\vec{u},\vec{v},i}))\leq 1$, $\op{deg}_x(C( b_{\vec{u},\vec{v}}))<f$, let $d_i\in\ZZ$ for $0\leq i\leq 2f-1$ be given by: 
\begin{equation}\label{d0}
    C(1+\Sigma_{i=0}^{f-1}x^i c_{\vec{u},\vec{v},i}+(1+x^f) b_{\vec{u},\vec{v}})=\Sigma_{i=0}^{2f-1}d_ix^i. 
\end{equation}
\begin{proposition}\label{11}
    We have
    \begin{equation}\label{d}
    d_i+d_{i+f}=0
    \end{equation}
    for $0\leq i\leq f-1$, and $d_i$ is determined by $u_{i-1},u_{i},v_{i-1},v_{i}$ as follows.\par
    For $i>0$, we have
\begin{table}[H]
    \centering
    \caption{\label{a}}
    \begin{tabular}{|cc|cccc|}
    \hline
    \multicolumn{2}{|c|}{\multirow{2}{*}{$d_i(i>0)$}}              & \multicolumn{4}{c|}{$(u_{i-1},u_i)$}                                                         \\ \cline{3-6} 
    \multicolumn{2}{|c|}{}                                         & \multicolumn{1}{c|}{$(0,0)$} & \multicolumn{1}{c|}{$(0,1)$} & \multicolumn{1}{c|}{$(1,0)$} & $(1,1)$ \\ \hline
    \multicolumn{1}{|c|}{\multirow{4}{*}{$(v_{i-1},v_i)$}} & $(0,0)$ & \multicolumn{1}{c|}{$0$}     & \multicolumn{1}{c|}{$-1$}    & \multicolumn{1}{c|}{$1$}     & $0$     \\ \cline{2-6} 
    \multicolumn{1}{|c|}{}                                 & $(0,1)$ & \multicolumn{1}{c|}{$-1$}    & \multicolumn{1}{c|}{$0$}     & \multicolumn{1}{c|}{$0$}     & $1$     \\ \cline{2-6} 
    \multicolumn{1}{|c|}{}                                 & $(1,0)$ & \multicolumn{1}{c|}{$1$}     & \multicolumn{1}{c|}{$-1$}    & \multicolumn{1}{c|}{$1$}     & $-1$    \\ \cline{2-6} 
    \multicolumn{1}{|c|}{}                                 & $(1,1)$ & \multicolumn{1}{c|}{$-1$}    & \multicolumn{1}{c|}{$1$}     & \multicolumn{1}{c|}{$-1$}    & $1 $    \\ \hline
    \end{tabular}
    \end{table}
For $i=0$ and $\overline{\rho}$ is reducible split, we have
\begin{table}[H]
    \centering
    \caption{\label{b}}
    \begin{tabular}{|cc|cccc|}
    \hline
    \multicolumn{2}{|c|}{\multirow{2}{*}{$d_0$}}              & \multicolumn{4}{c|}{$(u_{f-1}, u_0)$}                                                         \\ \cline{3-6} 
    \multicolumn{2}{|c|}{}                                         & \multicolumn{1}{c|}{$(0,0)$} & \multicolumn{1}{c|}{$(0,1)$} & \multicolumn{1}{c|}{$(1,0)$} & $(1,1)$ \\ \hline
    \multicolumn{1}{|c|}{\multirow{4}{*}{$(v_{f-1}, v_0)$}} & $(0,0)$ & \multicolumn{1}{c|}{$1$}     & \multicolumn{1}{c|}{$0$}    & \multicolumn{1}{c|}{$0$}     & $-1$     \\ \cline{2-6} 
    \multicolumn{1}{|c|}{}                                 & $(0,1)$ & \multicolumn{1}{c|}{$0$}    & \multicolumn{1}{c|}{$1$}     & \multicolumn{1}{c|}{$-1$}     & $0$   \\ \cline{2-6} 
    \multicolumn{1}{|c|}{}                                 & $(1,0)$ & \multicolumn{1}{c|}{$1$}     & \multicolumn{1}{c|}{$-1$}    & \multicolumn{1}{c|}{$1$}     & $-1$    \\ \cline{2-6} 
    \multicolumn{1}{|c|}{}                                 & $(1,1)$ & \multicolumn{1}{c|}{$-1$}    & \multicolumn{1}{c|}{$1$}     & \multicolumn{1}{c|}{$-1$}    & $1$     \\ \hline
    \end{tabular}
\end{table}
For $i=0$ and $\overline{\rho}$ is irreducible, we have
\begin{table}[H]
    \centering
    \caption{\label{c}}
    \begin{tabular}{|cc|cccc|}
    \hline
    \multicolumn{2}{|c|}{\multirow{2}{*}{$d_0$}}              & \multicolumn{4}{c|}{$(u_{f-1}, u_0)$}                                                         \\ \cline{3-6} 
    \multicolumn{2}{|c|}{}                                         & \multicolumn{1}{c|}{$(0,0)$} & \multicolumn{1}{c|}{$(0,1)$} & \multicolumn{1}{c|}{$(1,0)$} & $(1,1)$ \\ \hline
    \multicolumn{1}{|c|}{\multirow{4}{*}{$(v_{f-1}, v_0)$}} & $(0,0)$ & \multicolumn{1}{c|}{$1$}     & \multicolumn{1}{c|}{$0$}    & \multicolumn{1}{c|}{$0$}     & $-1$     \\ \cline{2-6} 
    \multicolumn{1}{|c|}{}                                 & $(0,1)$ & \multicolumn{1}{c|}{$0$}    & \multicolumn{1}{c|}{$1$}     & \multicolumn{1}{c|}{$-1$}     & $0$   \\ \cline{2-6} 
    \multicolumn{1}{|c|}{}                                 & $(1,0)$ & \multicolumn{1}{c|}{$0$}     & \multicolumn{1}{c|}{$0$}    & \multicolumn{1}{c|}{$0$}     & $0$   \\ \cline{2-6} 
    \multicolumn{1}{|c|}{}                                 & $(1,1)$ & \multicolumn{1}{c|}{$0$}    & \multicolumn{1}{c|}{$0$}     & \multicolumn{1}{c|}{$0$}    & $0$    \\ \hline
    \end{tabular}
\end{table}
\end{proposition}

\begin{proof}
 Note first $\op{deg}_x(C( c_{\vec{u},\vec{v},i}))\leq 1$, $\op{deg}_x (C(\lambda_{\vec{v},i}))\leq 1$, $\op{deg}_x (C(t_{\vec{u},\vec{v},i}))\leq 1$, and $\op{deg}_x(C( b_{\vec{u},\vec{v}}))<f$. By (\ref{b0}) and (\ref{d0}), if $0<i\leq f-1$, then we have
    \begin{equation}\label{d1}
        d_i=C_i(x^{i-1} c_{\vec{u},\vec{v},i-1}+x^{i} c_{\vec{u},\vec{v},i}-x^{i-1} t_{\vec{u},\vec{v},i-1}-x^{i} t_{\vec{u},\vec{v},i}-\frac{1}{2}(x^{i-1} \lambda_{\vec{v},i-1}+x^{i} \lambda_{\vec{v},i})),
    \end{equation}
    $$d_{i+f}=C_i(-x^{i-1} t_{\vec{u},\vec{v},i-1}-x^{i} t_{\vec{u},\vec{v},i}-\frac{1}{2}(x^{i-1} \lambda_{\vec{v},i-1}+x^{i} \lambda_{\vec{v},i})),$$
    and
    \begin{equation}\label{d2}
        \begin{split}
            d_i+d_{i+f}=&C_i(x^{i-1} c_{\vec{u},\vec{v},i-1}+x^{i} c_{\vec{u},\vec{v},i}-2(x^{i-1} t_{\vec{u},\vec{v},i-1}+x^{i} t_{\vec{u},\vec{v},i})-(x^{i-1} \lambda_{\vec{v},i-1}+x^{i} \lambda_{\vec{v},i}))                       
        \end{split}   
    \end{equation}
    The above sum is determined by $u_{i-1},u_{i},v_{i-1},v_{i}$ as follows.\par
    By definition of $\lambda_{\vec{v}}$, we have
    $$
    C_i(x^{i-1} \lambda_{\vec{v},i-1})=\left\{
        \begin{aligned}
            0 \text{, if }v_{i-1}=0\\
            1 \text{, if }v_{i-1}=1
        \end{aligned}
    \right., \quad
    C_i(x^{i} \lambda_{\vec{v},i})=\left\{
        \begin{aligned}
            0 \text{, if }(v_{i-1},v_i)=(0,0)\\
            -2 \text{, if }(v_{i-1},v_i)=(0,1)\\
            1 \text{, if }(v_{i-1},v_i)=(1,0)\\
            -3 \text{, if }(v_{i-1},v_i)=(1,1)\\
        \end{aligned}
    \right..
    $$
    By (\ref{5.0}), we have
        \begin{table}[H]
        \centering
        \begin{tabular}{|cl|c|c|}
        \hline
        \multicolumn{2}{|c|}{\multirow{2}{*}{$C_i(x^{i-1} c_{\vec{u},\vec{v},i-1})$}} & \multirow{2}{*}{$u_{i-1}=0$} & \multirow{2}{*}{$u_{i-1}=1$} \\
        \multicolumn{2}{|c|}{}                    &                              &                              \\ \hline
        \multicolumn{2}{|c|}{$v_{i-1}=0$}         & $0$                            & $1$                            \\ \hline
        \multicolumn{2}{|c|}{$v_{i-1}=1$}         & $1$                            & $0$                            \\ \hline
        \end{tabular}                          
    \end{table}
    \begin{table}[H]
        \centering
        \begin{tabular}{|cc|cccc|}
            \hline
            \multicolumn{2}{|c|}{\multirow{2}{*}{$C_i(x^{i} c_{\vec{u},\vec{v},i})$}}              & \multicolumn{4}{c|}{$(u_{i-1}, u_i)$}                                                         \\ \cline{3-6} 
            \multicolumn{2}{|c|}{}                                         & \multicolumn{1}{c|}{$(0,0)$} & \multicolumn{1}{c|}{$(0,1)$} & \multicolumn{1}{c|}{$(1,0)$} & $(1,1)$ \\ \hline
            \multicolumn{1}{|c|}{\multirow{4}{*}{$(v_{i-1}, v_i)$}} & $(0,0)$ & \multicolumn{1}{c|}{$0$}     & \multicolumn{1}{c|}{$-2$}    & \multicolumn{1}{c|}{$1$}     & $-1$     \\ \cline{2-6} 
            \multicolumn{1}{|c|}{}                                 & $(0,1)$ & \multicolumn{1}{c|}{$-2$}    & \multicolumn{1}{c|}{$0$}     & \multicolumn{1}{c|}{$-1$}     & $1$   \\ \cline{2-6} 
            \multicolumn{1}{|c|}{}                                 & $(1,0)$ & \multicolumn{1}{c|}{$1$}     & \multicolumn{1}{c|}{$-3$}    & \multicolumn{1}{c|}{$2$}     & $-2$    \\ \cline{2-6} 
            \multicolumn{1}{|c|}{}                                 & $(1,1)$ & \multicolumn{1}{c|}{$-3$}    & \multicolumn{1}{c|}{$1$}     & \multicolumn{1}{c|}{$-2$}    & $2 $    \\ \hline
        \end{tabular}
            
    \end{table}
      By (\ref{6.0}), we have
    \begin{table}[H]
        \centering

            \begin{tabular}{|cl|c|c|}
                \hline
                \multicolumn{2}{|c|}{\multirow{2}{*}{$C_i(x^{i-1} t_{\vec{u},\vec{v},i-1})$}} & \multirow{2}{*}{$u_{i-1}=0$} & \multirow{2}{*}{$u_{i-1}=1$} \\
                \multicolumn{2}{|c|}{}                    &                              &                              \\ \hline
                \multicolumn{2}{|c|}{$v_{i-1}=0$}         & $0$                            & $1$                            \\ \hline
                \multicolumn{2}{|c|}{$v_{i-1}=1$}         & $0$                            & $0$                            \\ \hline
                \end{tabular}

    \end{table}

    \begin{table}[H]
        \centering
        \begin{tabular}{|cc|cccc|}
            \hline
            \multicolumn{2}{|c|}{\multirow{2}{*}{$C_i(x^{i} t_{\vec{u},\vec{v},i})$}}              & \multicolumn{4}{c|}{$(u_{i-1}, u_i)$}                                                         \\ \cline{3-6} 
            \multicolumn{2}{|c|}{}                                         & \multicolumn{1}{c|}{$(0,0)$} & \multicolumn{1}{c|}{$(0,1)$} & \multicolumn{1}{c|}{$(1,0)$} & $(1,1)$ \\ \hline
            \multicolumn{1}{|c|}{\multirow{4}{*}{$(v_{i-1}, v_i)$}} & $(0,0)$ & \multicolumn{1}{c|}{$0$}     & \multicolumn{1}{c|}{$-1$}    & \multicolumn{1}{c|}{$0$}     & $-1$     \\ \cline{2-6} 
            \multicolumn{1}{|c|}{}                                 & $(0,1)$ & \multicolumn{1}{c|}{$0$}    & \multicolumn{1}{c|}{$1$}     & \multicolumn{1}{c|}{$0$}     & $1$   \\ \cline{2-6} 
            \multicolumn{1}{|c|}{}                                 & $(1,0)$ & \multicolumn{1}{c|}{$0$}     & \multicolumn{1}{c|}{$-2$}    & \multicolumn{1}{c|}{$0$}     & $-2$    \\ \cline{2-6} 
            \multicolumn{1}{|c|}{}                                 & $(1,1)$ & \multicolumn{1}{c|}{$0$}    & \multicolumn{1}{c|}{$2$}     & \multicolumn{1}{c|}{$0$}    & $2$     \\ \hline
\end{tabular}
            
    \end{table}
    Combining (\ref{d2}) and tables above, we have $d_i+d_{i+f}=0$ for $i>0$. Combining (\ref{d1}) and tables above, we have Table \ref{a}.\par
    If $i=0$, then we have
    \begin{equation}\label{d3}
        \begin{split}    
        d_0&=C_0\Big(1+ c_{\vec{u},\vec{v},0}- t_{\vec{u},\vec{v},0}-S\left(x^{f-1} t_{\vec{u},\vec{v},f-1}\right)-\left(u_0u_{f-1}+\left(1-u_0\right)\left(1-u_{f-1}\right)\right)\\ &-\frac{1}{2}\left(x^{0} \lambda_{\vec{v},0}+S\left(x^{f-1} \lambda_{\vec{v},f-1}\right)\right)\Big),  
        \end{split}
    \end{equation}
    
    \begin{equation*}
        \begin{split}
        d_f=&C_f\left(x^{f-1} c_{\vec{u},\vec{v},f-1}\right)+C_0\Big(-t_{\vec{u},\vec{v},0}-S\left(x^{f-1} t_{\vec{u},\vec{v},f-1}\right)-\left(u_0u_{f-1}+\left(1-u_0\right)\left(1-u_{f-1}\right)\right)\\
        &-\frac{1}{2}\left(x^{0} \lambda_{\vec{v},0}+S\left(x^{f-1} \lambda_{\vec{v},f-1}\right)\right)\Big),
        \end{split}
    \end{equation*}
    and
    \begin{equation}\label{d4}
        \begin{split}
    d_0+d_f=&C_0\left(1+ c_{\vec{u},\vec{v},0}-2\left( t_{\vec{u},\vec{v},0}-S\left(x^{f-1} t_{\vec{u},\vec{v},f-1}\right)\right)-2\left(u_0u_{f-1}+\left(1-u_0\right)\left(1-u_{f-1}\right)\right)\right.\\
    &\left.-\left(x^{0} \lambda_{\vec{v},0}+S\left(x^{f-1} \lambda_{\vec{v},f-1}\right)\right)\right)+C_f\left(x^{f-1} c_{\vec{u},\vec{v},f-1}\right).
    \end{split}
    \end{equation}
  The above sum is determined by $u_{f-1},u_{0},v_{f-1},v_{0}$ as follows. 
    $$u_0u_{f-1}+(1-u_0)(1-u_{f-1})=\left\{
        \begin{aligned}
            1 \text{, if }(u_{f-1},u_0)=(0,0)\\
            0 \text{, if }(u_{f-1},u_0)=(0,1)\\
            0 \text{, if }(u_{f-1},u_0)=(1,0)\\
            1 \text{, if }(u_{f-1},u_0)=(1,1)\\
        \end{aligned}
    \right.$$
    If $\overline{\rho}$ is reducible split, by definition of $\lambda_{\vec{v}}$, we have
    $$
    C_0(S(x^{f-1} \lambda_{\vec{v},f-1}))=\left\{
        \begin{aligned}
            0 \text{, if }v_{f-1}=0\\
            1 \text{, if }v_{f-1}=1
        \end{aligned}
    \right., \quad 
    C_0(x^{0} \lambda_{\vec{v},0})=\left\{
        \begin{aligned}
            0 \text{, if }(v_{f-1},v_0)=(0,0)\\
            -2 \text{, if }(v_{f-1},v_0)=(0,1)\\
            1 \text{, if }(v_{f-1},v_0)=(1,0)\\
            -3 \text{, if }(v_{f-1},v_0)=(1,1)\\
        \end{aligned}
    \right..
    $$
    By (\ref{5.0}), we have
    \begin{table}[H]
        \centering
        
            \begin{tabular}{|cl|c|c|}
                \hline
                \multicolumn{2}{|c|}{\multirow{2}{*}{$C_f(x^{f-1} c_{\vec{u},\vec{v},f-1})$}} & \multirow{2}{*}{$u_{f-1}=0$} & \multirow{2}{*}{$u_{f-1}=1$} \\
                \multicolumn{2}{|c|}{}                    &                              &                              \\ \hline
                \multicolumn{2}{|c|}{$v_{f-1}=0$}         & $0$                            & $1$                            \\ \hline
                \multicolumn{2}{|c|}{$v_{f-1}=1$}         & $1$                            &$ 0$                            \\ \hline
            \end{tabular}           
        
    \end{table}
    \begin{table}[H]
        \centering
        \begin{tabular}{|cc|cccc|}
            \hline
            \multicolumn{2}{|c|}{\multirow{2}{*}{$C_0( c_{\vec{u},\vec{v},0})$}}              & \multicolumn{4}{c|}{$(u_{f-1}, u_0)$}                                                         \\ \cline{3-6} 
            \multicolumn{2}{|c|}{}                                         & \multicolumn{1}{c|}{$(0,0)$} & \multicolumn{1}{c|}{$(0,1)$} & \multicolumn{1}{c|}{$(1,0)$} & $(1,1)$ \\ \hline
            \multicolumn{1}{|c|}{\multirow{4}{*}{$(v_{f-1}, v_0)$}} & $(0,0)$ & \multicolumn{1}{c|}{$1$}     & \multicolumn{1}{c|}{$-1$}    & \multicolumn{1}{c|}{$0$}     & $-2$     \\ \cline{2-6} 
            \multicolumn{1}{|c|}{}                                 & $(0,1)$ & \multicolumn{1}{c|}{$-1$}    & \multicolumn{1}{c|}{$1$}     & \multicolumn{1}{c|}{$-2$}     & $0$   \\ \cline{2-6} 
            \multicolumn{1}{|c|}{}                                 & $(1,0)$ & \multicolumn{1}{c|}{$2$}     & \multicolumn{1}{c|}{$-2$}    & \multicolumn{1}{c|}{$1$}     & $-3$   \\ \cline{2-6} 
            \multicolumn{1}{|c|}{}                                 & $(1,1)$ & \multicolumn{1}{c|}{$-2$}    & \multicolumn{1}{c|}{$2$}     & \multicolumn{1}{c|}{$-3$}    & $1$    \\ \hline
        \end{tabular}
              \end{table}
    By (\ref{6.0}), we have
    \begin{table}[H]
        \centering
        
            \begin{tabular}{|cl|c|c|}
                \hline
                \multicolumn{2}{|c|}{\multirow{2}{*}{$C_0(S(x^{f-1} t_{\vec{u},\vec{v},f-1}))$}} & \multirow{2}{*}{$u_{f-1}=0$} & \multirow{2}{*}{$u_{f-1}=1$} \\
                \multicolumn{2}{|c|}{}                    &                              &                              \\ \hline
                \multicolumn{2}{|c|}{$v_{f-1}=0$}         & $0$                            & $1$                            \\ \hline
                \multicolumn{2}{|c|}{$v_{f-1}=1$}         & $0$                            & $0$                            \\ \hline
            \end{tabular}

    \end{table}
    \begin{table}[H]
        \centering
        \begin{tabular}{|cc|cccc|}
            \hline
            \multicolumn{2}{|c|}{\multirow{2}{*}{$C_0( t_{\vec{u},\vec{v},0})$}}              & \multicolumn{4}{c|}{$(u_{f-1}, u_0)$}                                                         \\ \cline{3-6} 
            \multicolumn{2}{|c|}{}                                         & \multicolumn{1}{c|}{$(0,0)$} & \multicolumn{1}{c|}{$(0,1)$} & \multicolumn{1}{c|}{$(1,0)$} & $(1,1)$ \\ \hline
            \multicolumn{1}{|c|}{\multirow{4}{*}{$(v_{f-1}, v_0)$}} & $(0,0)$ & \multicolumn{1}{c|}{$0$}     & \multicolumn{1}{c|}{$0$}    & \multicolumn{1}{c|}{$0$}     & $-2$     \\ \cline{2-6} 
            \multicolumn{1}{|c|}{}                                 & $(0,1)$ & \multicolumn{1}{c|}{$0$}    & \multicolumn{1}{c|}{$2$}     & \multicolumn{1}{c|}{$0$}     & $0$   \\ \cline{2-6} 
            \multicolumn{1}{|c|}{}                                 & $(1,0)$ & \multicolumn{1}{c|}{$0$}     & \multicolumn{1}{c|}{$-1$}    & \multicolumn{1}{c|}{$0$}     & $-3$   \\ \cline{2-6} 
            \multicolumn{1}{|c|}{}                                 & $(1,1)$ & \multicolumn{1}{c|}{$0$}    & \multicolumn{1}{c|}{$3$}     & \multicolumn{1}{c|}{$0$}    & $1$    \\ \hline
            \end{tabular} 
            
    \end{table}
    Combining (\ref{d3}), (\ref{d4}) and tables above, we get Table \ref{b} and $d_0+d_f=0$ for reducible split $\overline{\rho}$. \par
    If $\overline{\rho}$ is irreducible, by definition of $\lambda_{\vec{v}}$, we have
    $$
    C_0(S(x^{f-1} \lambda_{\vec{v},f-1}))=\left\{
        \begin{aligned}
            0 \text{, if }v_{f-1}=0\\
            1 \text{, if }v_{f-1}=1
        \end{aligned}
    \right., \quad
    C_0(x^{0} \lambda_{\vec{v},0})=\left\{
        \begin{aligned}
            0 \text{, if }(v_{f-1},v_0)=(0,0)\\
            -2 \text{, if }(v_{f-1},v_0)=(0,1)\\
            -1 \text{, if }(v_{f-1},v_0)=(1,0)\\
            -1 \text{, if }(v_{f-1},v_0)=(1,1)\\
        \end{aligned}
    \right..
    $$
    
    By (\ref{5.0}), we have
    
    \begin{table}[H]
        \centering

            \begin{tabular}{|cl|c|c|}
                \hline
                \multicolumn{2}{|c|}{\multirow{2}{*}{$C_f(x^{f-1} c_{\vec{u},\vec{v},f-1})$}} & \multirow{2}{*}{$u_{f-1}=0$} & \multirow{2}{*}{$u_{f-1}=1$} \\
                \multicolumn{2}{|c|}{}                    &                              &                              \\ \hline
                \multicolumn{2}{|c|}{$v_{f-1}=0$}         & $0$                            & $1$                            \\ \hline
                \multicolumn{2}{|c|}{$v_{f-1}=1$}         & $1$                            & $0$                            \\ \hline
            \end{tabular}

    \end{table}
    \begin{table}[H]
        \centering
        \begin{tabular}{|cc|cccc|}
            \hline
            \multicolumn{2}{|c|}{\multirow{2}{*}{$C_0( c_{\vec{u},\vec{v},0})$}}              & \multicolumn{4}{c|}{$(u_{f-1}, u_0)$}                                                         \\ \cline{3-6} 
            \multicolumn{2}{|c|}{}                                         & \multicolumn{1}{c|}{$(0,0)$} & \multicolumn{1}{c|}{$(0,1)$} & \multicolumn{1}{c|}{$(1,0)$} & $(1,1)$ \\ \hline
            \multicolumn{1}{|c|}{\multirow{4}{*}{$(v_{f-1}, v_0)$}} & $(0,0)$ & \multicolumn{1}{c|}{$1$}     & \multicolumn{1}{c|}{$-1$}    & \multicolumn{1}{c|}{$0$}     & $-2$     \\ \cline{2-6} 
            \multicolumn{1}{|c|}{}                                 & $(0,1)$ & \multicolumn{1}{c|}{$-1$}    & \multicolumn{1}{c|}{$1$}     & \multicolumn{1}{c|}{$-2$}     & $0$   \\ \cline{2-6} 
            \multicolumn{1}{|c|}{}                                 & $(1,0)$ & \multicolumn{1}{c|}{$0$}     & \multicolumn{1}{c|}{$0$}    & \multicolumn{1}{c|}{$-1$}     & $-1$   \\ \cline{2-6} 
            \multicolumn{1}{|c|}{}                                 & $(1,1)$ & \multicolumn{1}{c|}{$0$}    & \multicolumn{1}{c|}{$0$}     & \multicolumn{1}{c|}{$-1$}    & $-1$    \\ \hline
            \end{tabular}
            
    \end{table}
    By (\ref{6.0}), we have
    \begin{table}[H]
        \centering
        
            \begin{tabular}{|cl|c|c|}
                \hline
                \multicolumn{2}{|c|}{\multirow{2}{*}{$C_0(S(x^{f-1} t_{\vec{u},\vec{v},f-1}))$}} & \multirow{2}{*}{$u_{f-1}=0$} & \multirow{2}{*}{$u_{f-1}=1$} \\
                \multicolumn{2}{|c|}{}                    &                              &                              \\ \hline
                \multicolumn{2}{|c|}{$v_{f-1}=0$}         & $0$                            & $1$                            \\ \hline
                \multicolumn{2}{|c|}{$v_{f-1}=1$}         & $0$                            & $0$                            \\ \hline
            \end{tabular}

    \end{table}
    \begin{table}[H]
        \centering
        \begin{tabular}{|cc|cccc|}
            \hline
            \multicolumn{2}{|c|}{\multirow{2}{*}{$C_0( t_{\vec{u},\vec{v},0})$}}              & \multicolumn{4}{c|}{$(u_{f-1}, u_0)$}                                                         \\ \cline{3-6} 
            \multicolumn{2}{|c|}{}                                         & \multicolumn{1}{c|}{$(0,0)$} & \multicolumn{1}{c|}{$(0,1)$} & \multicolumn{1}{c|}{$(1,0)$} & $(1,1)$ \\ \hline
            \multicolumn{1}{|c|}{\multirow{4}{*}{$(v_{f-1}, v_0)$}} & $(0,0)$ & \multicolumn{1}{c|}{$0$}     & \multicolumn{1}{c|}{$0$}    & \multicolumn{1}{c|}{$0$}     & $-2$     \\ \cline{2-6} 
            \multicolumn{1}{|c|}{}                                 & $(0,1)$ & \multicolumn{1}{c|}{$0$}    & \multicolumn{1}{c|}{$2$}     & \multicolumn{1}{c|}{$0$}     & $0$   \\ \cline{2-6} 
            \multicolumn{1}{|c|}{}                                 & $(1,0)$ & \multicolumn{1}{c|}{$0$}     & \multicolumn{1}{c|}{$1$}    & \multicolumn{1}{c|}{$0$}     & $-1$   \\ \cline{2-6} 
            \multicolumn{1}{|c|}{}                                 & $(1,1)$ & \multicolumn{1}{c|}{$0$}    & \multicolumn{1}{c|}{$1$}     & \multicolumn{1}{c|}{$0$}    & $-1$    \\ \hline
            \end{tabular}   
            
    \end{table}
    Combining (\ref{d3}), (\ref{d4}) and tables above, we get Table \ref{c} and $d_0+d_f=0$ for irreducible $\overline{\rho}$. 
    \end{proof}
\subsection{Explicit description of $W_{D}(\overline{\rho})$ for generic $\overline{\rho}$}\label{s32}
We now give explicit description of $W_{D}(\overline{\rho}).$ Let $\overline{\psi}=\overline{\iota}^{(q+1)b+1+c}\in W_{D}(\overline{\rho})$ with $0\leq b\leq q-2,$ $0\leq c\leq q-1$. By (\ref{3.6}) and \eqref{d}, we have
\begin{equation*}
    \begin{split}
        (q+1)b+1+c\equiv &(L(1+\Sigma_{i=0}^{f-1}x^i c_{\vec{u},\vec{v},i}+(1+x^f) b_{\vec{u},\vec{v}})(r_0,...,r_{f-1},p)\\
        &+C(1+\Sigma_{i=0}^{f-1}x^i c_{\vec{u},\vec{v},i}+(1+x^f) b_{\vec{u},\vec{v}})(r_0,...,r_{f-1},p))~(\op{mod}q^2-1)\\
        \equiv&\sum_{i=0}^{f-1} q^{w_i}p^ir_i+\sum_{i=0}^{2f-1} d_ip^i~ (\op{mod}q^2-1)\\
        \equiv&\sum_{i=0}^{f-1} q^{w_i}p^ir_i+\sum_{i=0}^{f-1} (d_i+qd_{i+f})p^i ~(\op{mod}q^2-1)\\
        \equiv&\sum_{i=0}^{f-1} q^{w_{i}} p^{i} r_{i}+(1-q)\sum_{i=0}^{f-1}d_ip^i~ (\op{mod}q^2-1).
    \end{split}
\end{equation*}

\begin{theorem}\label{26}
    Let $\overline{\rho}$ be generic reducible split as in Lemma \ref{7}, then $\overline{\psi}=\overline{\iota}^{(q+1)b+1+c}\in W_{D}(\overline{\rho})$, if and only if
    $$(q+1)b+1+c\equiv\sum_{i=0}^{f-1} q^{w_{i}} p^{i} r_{i}+(1-q)\sum_{i=0}^{f-1}d_ip^i ~(\op{mod}q^2-1)$$
    where $w_{i} \in\{0,1\}$ and $d_{i} \in\{-1,0,1\}$ satisfy the following relations.
    \begin{itemize}
        \item For $i>0$, 
        $w_i=\left\{\begin{array}{cl}
        1&  \text { if } d_i=-1\\
        0&  \text { if } d_i=1             
        \end{array}\right.$, 
    and if $d_i=0$, then $(w_{i-1},w_i)=(0,0)$ or $(1,1)$.
        \item 
        $w_0=\left\{\begin{array}{cl}
        1&  \text { if } d_0=-1\\
        0&  \text { if } d_0=1             
        \end{array}\right.$, 
    and if $d_0=0$, then $(w_{f-1},w_0)=(0,1)$ or $(1,0)$. 
    \end{itemize}
Note that, from the above relations, if $d_i=0$ for all $0\leq i\leq f-1$, $\vec{w}$ does not exist, and if $(d_0,\ldots,d_{f-1})\neq(0,\ldots,0)$, there is exactly one $\vec{w}$.  We write $\overline{\psi}_{\vec{w},\vec{d}}$ for such $\overline{\psi}$ and we have 
\begin{equation}\label{w}
    \overline{\psi}_{\vec{w},\vec{d}}=\overline{\psi}_{\vec{w}',\vec{d}'}\Longleftrightarrow \vec{w}=\vec{w}',\ \vec{d}=\vec{d}'.
\end{equation}
As a consequence, 
$$W_D(\overline\rho)=\{\overline{\psi}_{\vec{w},\vec{d}}\mid \vec{d}\in\{-1,0,1\}^{\ZZ/f\ZZ},\ \vec{d}\neq (0,...,0)\},$$
where $\vec{w}$ is uniquely determined by $\vec{d}$. We have $\# W_{D}(\overline{\rho})=3^{f}-1$.  
\end{theorem}
\begin{proof}
    The relations of $\vec{w}$ and $\vec{d}$ follow from Proposition \ref{19}, Tables \ref{a} and \ref{b} in Proposition \ref{11}. We need to prove $\vec{w}$ is uniquely determined by $\vec{d}\neq (0,...,0)$ and construct $\vec{u}$ and $\vec{v}$ for each $\vec{d}\neq (0,...,0)$.\par
If $d_i=0$ for all $0\leq i\leq f-1$, we deduce from the first relation that all $w_i$ should be the same, while the second relation tells us that $w_0$ and $w_{f-1}$
are different. This is a contradiction.\par 
If $\vec{d}\neq 0$, we will show the uniqueness of $\vec{w}$. By the two relations, we have if $d_i=-1$ or $1$ , $w_i$ is determined by $d_i$; if $d_i=0$, $w_i$ is determined by the nearest $d_j\neq 0$ $(j<i)$ if there exists $j<i$, $d_j\neq 0$, and by the farthest $d_j\neq 0$ $(j>i)$ if $d_j=0$ for all $j<i$.\par 
For $\vec{d}\neq (0,...,0)$, to show $\overline{\psi}_{\vec{w},\vec{d}}\in W_D(\overline\rho)$, by Lemma \ref{P}, it suffices to find $\vec{u},\vec{v}\in\{0,1\}^{\ZZ/f\ZZ}$ such that $\overline{\psi}_{\vec{w},\vec{d}}=\overline{\psi}_{\vec{u},\vec{v}}$. 
Let $v_{i-1}\equiv d_i(\text{mod }2)$ for $0\leq i\leq f-1$, we fix the choice of $\vec{v}$, and construct $\vec{u}$. By Table \ref{a} and Table \ref{b}, if $d_i=-1$ or $1$, $u_i$ is determined by $(v_{i-1},v_i,d_i)$, and if $d_i=0$, $u_i$ is determined by $(v_{i-1},v_i,u_{i-1})$. \par
If $\overline{\psi}_{\vec{w},\vec{d}}=\overline{\psi}_{\vec{w}',\vec{d}'}$, then 
\begin{equation*}   
    0 \equiv(q-1)\sum_{i=0}^{f-1}p^i((w_i-w'_{i})r_i-(d_i-d'_{i}))\ (\op{mod}q^2-1).  
\end{equation*}
Note that $0\leq r_i\leq p-3$ for all $i$ since $\overline\rho$ is generic. Since 
$\sum_{i=0}^{f-1}p^i\left|(w_i-w'_{i})r_i-(d_i-d'_{i})\right|\leq \sum_{i=0}^{f-1}p^i(r_i+2)< q+1$, we have $\sum_{i=0}^{f-1}p^i((w_i-w'_{i})r_i-(d_i-d'_{i}))=0$. 
Since $\left|(w_i-w'_{i})r_i-(d_i-d'_{i})\right|\leq p-1$ for all $i$, we have $(w_i-w'_{i})r_i-(d_i-d'_{i})=0$ for all $i$. Since $w_i\leq w'_i$ if $d_i>d'_i$ by the relations above, we have $d_i=d'_i$ for all $i$. 
Then $\vec{w}=\vec{w}'$ since $\vec{w}$ is uniquely determined by $\vec{d}$. 
\end{proof}

\begin{definition}
    Let $\overline{\rho}$ be generic reducible split as in Lemma \ref{7}.  
    For each $f$-tuple ${\vec{v}}\in\{0,1\}^{\ZZ/f\ZZ}$, 
    we define a set $U_{\vec{v}}$ consisting of $f$-tuples ${\vec{u}}\in\{0,1\}^{\ZZ/f\ZZ}$ such that: 
    \begin{itemize}
        \item If $v_{f-1}=1$, then $u_{f-1}-u_0\equiv v_{f-1}-v_0+1\ (\op{mod}2)$;
        \item if $0<i\leq f-1$ and $v_{i-1}=1$, then $u_{i-1}-u_i\equiv v_{i-1}-v_i\ (\op{mod}2)$.
    \end{itemize}    
    For ${\vec{v}}\in\{0,1\}^{\ZZ/f\ZZ}$, 
    we define: 
    \begin{equation}
        W_{D}^{\vec{v}}(\overline{\rho}):=\{\overline{\psi}: l^{\times} \rightarrow \FF^{\times}\mid\exists \vec{u}\in U_{\vec{v}}, \text{ s.t. } \sigma_{\vec{v}}(\overline{\rho})=\overline{\Theta}([\overline{\psi}])_{\vec{u}}\}, 
    \end{equation}
    where $\sigma_{\vec{v}}(\overline{\rho})=\lambda_{\vec{v}}\left(r_{0}, \ldots, r_{f-1},p\right) \otimes\left(\overline{\kappa}_0 \circ \operatorname{det}^{e(\lambda_{\vec{v}})\left(r_{0}, \ldots, r_{f-1},p\right)}\right)$, and define: 
    $$W_{D,\vec{v}}(\overline{\rho}):=\{\overline\psi:l^{\times} \rightarrow \FF^{\times}\mid\exists\vec{u}\in\{0,1\}^{\ZZ/f\ZZ},\text{ s.t. } \overline\psi=\overline\psi_{\vec{u},\vec{v}}\}.$$
\end{definition}
\begin{proposition}\label{27}
    We have the following properties:
    \begin{enumerate}[label=(\roman*)]
        \item If $\ell(\lambda_{\vec{v}})=f$, $W_{D}^{\vec{v}}(\overline{\rho})=\emptyset$.
        \item If $\ell(\lambda_{\vec{v}})=d<f$, $\#W_{D}^{\vec{v}}(\overline{\rho})=2^{f-d}$.
        \item Assume $\overline{\psi}=\overline{\psi}_{\vec{u},\vec{v}}\in W_{D}^{\vec{v}}(\overline{\rho})$, $\vec{u}\in U_{\vec{v}}$. If there exist $\vec{u}'$ and ${\vec{v}}'$, such that $\overline{\Theta}([\overline\psi])_{\vec{u}'}=\sigma_{\vec{v}'}(\overline{\rho})$, then ${\vec{v}}\leq {{\vec{v}}'}$. 
        In particular, if $W_{D}^{\vec{v}}(\overline{\rho})\cap W_{D}^{\vec{v}'}(\overline{\rho})\neq \emptyset$, then ${\vec{v}}= {{\vec{v}}'}$. 
        \item $W_{D}(\overline{\rho})=\bigsqcup\limits_{{\vec{v}}} W_{D}^{\vec{v}}(\overline{\rho})$. 
        \item $W_{D}^{\vec{v}}(\overline{\rho})=W_{D,\vec{v}}(\overline{\rho})\backslash\bigcup_{\vec{v}'<\vec{v}}W_{D,\vec{v}'}(\overline{\rho})$.
    \end{enumerate}
\end{proposition}
\begin{proof}
    For (i), since $\ell(\lambda_{\vec{v}})=f$, we have $v_i=1$ for $0\leq i\leq f-1$, then $u_{f-1}-u_0\equiv 1\ (\op{mod}2)$ and $u_{i-1}-u_i\equiv 0\ (\op{mod}2)$ for $0<i\leq f-1$. It is easy to see that $U_{\vec{v}}=\emptyset$.  
    For (ii), there exists $i$, such that $v_i=0$. Let $u_{i+1}=0$ or $1$. Then we determine $u_{j}$ for $j=i+2,...,f-1,0,...,i$ in turn. If $v_j=0$, there are two choices for $u_{j+1}$, and if $v_j=1$, $u_{j+1}$ is uniquely determined by $u_{j}$, 
    so there are $2^{f-d}$ choices for $\vec{u}$. Since $\overline{\psi}_{\vec{u},\vec{v}}=\overline{\psi}_{\vec{u}',\vec{v}}$ if and only if $\vec{u}=\vec{u}'$ by $w_i\equiv u_i+v_i\ (\op{mod}2)$ and (\ref{w}), we have $\#W_{D}^{\vec{v}}(\overline{\rho})=2^{f-d}$. \par
    Next we prove (iii). We need to prove: for $0\leq i\leq f-1$, if $v_i=1$, then $v'_i=1$. Since $\overline\psi\in W_D(\overline{\rho})$, we may suppose $\overline\psi=\overline\psi_{\vec{w},\vec{d}}$ by Theorem \ref{26} where $\vec{w},\vec{d}$ are determined by $\overline\psi$.
    Since $v_i=1$, $d_{i+1}$ can be $-1$ or $1$ by Table \ref{a} and Table \ref{b}. We will show that if $v_i'=0$, then $w_i\neq w'_i$, and we deduce a contradiction. We take $d_{i+1}=-1$ and $i\neq f-1$ as an example. The other cases can be treated similarly. 
    By $u_i-u_{i+1}\equiv v_i-v_{i+1}\ (\op{mod}2)$, Table \ref{a} and Table \ref{b}, we have $u_i=0$. 
    If $v'_i=0$, we have $u'_i=0$ by $d_{i+1}=-1$ and Table \ref{a}. Then $w_i\equiv u_i+v_i\equiv 1\ (\op{mod}2)$ and $w'_i\equiv u'_i+v'_i\equiv 0\ (\op{mod}2)$, which is a contradiction. \par
    If there exists $\overline\psi\in W_{D}^{\vec{v}}(\overline{\rho})\cap W_{D}^{\vec{v}'}(\overline{\rho})$, then there exist $\vec{u}$ and $\vec{v}$ such that $\overline\psi=\overline\psi_{\vec{u},\vec{v}}$ since $\overline\psi\in W_{D}^{\vec{v}}(\overline{\rho})$. Since $\overline\psi_{\vec{u},\vec{v}}\in W_{D}^{\vec{v}'}(\overline{\rho})$, we have $\vec{v}'\leq \vec{v}$. Similarly we also have $\vec{v}\leq \vec{v}'$, then $\vec{v}= \vec{v}'$. 
    We deduce that $W_{D}^{\vec{v}}(\overline{\rho})\cap W_{D}^{\vec{v}'}(\overline{\rho})=\emptyset$ if $\vec{v}\neq\vec{v}'$. By (ii) and Theorem \ref{26}, we have $\#\bigsqcup\limits_{{\vec{v}}} W_{D}^{\vec{v}}(\overline{\rho})=\sum_{d=0}^{f-1}\tbinom{f}{d}2^{f-d}=3^{f}-1=\#W_{D}(\overline{\rho})$ which implies (iv). \par
    By (iii) we have $W_{D}^{\vec{v}}(\overline{\rho})\subseteq W_{D,\vec{v}}(\overline{\rho})\backslash\bigcup\limits_{\vec{v}'<\vec{v}}W_{D,\vec{v}'}(\overline{\rho})$. For $\overline\psi\in W_{D,\vec{v}}(\overline{\rho})\backslash\bigcup\limits_{\vec{v}'<\vec{v}}W_{D,\vec{v}'}(\overline{\rho})$, by (iv) there exists $\vec{v}''$ such that $\overline\psi\in W_{D}^{\vec{v}''}(\overline{\rho})$. By (iii), we have $\vec{v}''\leq\vec{v}$. Since $\overline\psi\notin \bigcup_{\vec{v}'<\vec{v}}W_{D,\vec{v}'}(\overline{\rho})$, 
    we have $\vec{v}''=\vec{v}$, i.e. $\overline\psi\in W_{D}^{\vec{v}}(\overline{\rho})$. 
\end{proof}
The following theorem determines $W_D(\overline{\rho})$ for non-semisimple generic $\overline{\rho}$. 

\begin{theorem}\label{thm::nonsplit}
    Let $\overline{\rho}$ be generic reducible nonsplit as in Definition \ref{2.6} and let $\vec{v}(\overline{\rho})\in\{0,\ldots,f-1\}^{\ZZ/f\ZZ}$ be the $f$-tuple associated to $\overline{\rho}.$ Then
$$W_{D}(\overline{\rho})=\bigsqcup_{{\vec{v}}\leq\vec{v}(\overline{\rho})}W_{D}^{{\vec{v}}}(\overline{\rho}^{{\rm ss}}).$$
Note that each $W_{D}^{{\vec{v}}}(\overline{\rho}^{{\rm ss}})$ is explicit.
\end{theorem}
\begin{proof}
    By Lemma \ref{18}, we have $W_{D}(\overline{\rho})=\bigcup\limits_{\vec{v}\leq\vec{v}(\overline{\rho})}W_{D,\vec{v}}(\overline{\rho}^{{\rm ss}})$. By (iv) and (v) of Proposition \ref{27}, we have $\bigcup\limits_{\vec{v}\leq\vec{v}(\overline{\rho})}W_{D,\vec{v}}(\overline{\rho}^{{\rm ss}})=\bigsqcup\limits_{{\vec{v}}\leq\vec{v}(\overline{\rho})}W_{D}^{{\vec{v}}}(\overline{\rho}^{{\rm ss}})$. 
\end{proof}

\begin{remark}\label{rk::number of serre wt}
 If $\overline{\rho}$ is reducible nonsplit, i.e. $\ell(\lambda_{\vec{v}(\overline{\rho})})=d<f$, by Proposition \ref{27}(ii), $\# W_{D}(\overline{\rho})= \sum_{k=0}^{d} \tbinom{d}{k} 2^{f-k} = 3^d2^{f-d}.$ Note that this number also appears in  \cite[Prop. 14.7]{Breuil-Paskunas}.
\end{remark}

\begin{theorem}\label{thm::irreducible case}
    Let $\overline{\rho}$ be generic irreducible as in Lemma \ref{16}, then $\overline{\psi}=\overline{\iota}^{(q+1)b+1+c}\in W_{D}(\overline{\rho})$, if and only if
    $$(q+1)b+1+c\equiv \sum_{i=0}^{f-1} q^{w_{i}} p^{i} r_{i}+(1-q)\sum_{i=0}^{f-1}d_ip^i ~(\op{mod}q^2-1)$$
    where $w_{i} \in\{0,1\}$ and $d_{i} \in\{-1,0,1\}$ satisfy the following relations.
    \begin{itemize}
        \item For $i>0$, 
            $w_i=\left\{\begin{array}{cl}
            1&  \text { if } d_i=-1\\
            0&  \text { if } d_i=1             
            \end{array}\right.$, 
        and if $d_i=0$, then $(w_{i-1},w_i)=(0,0)$ or $(1,1)$.
        
        \item 
            $(w_{f-1},w_0)=\left\{\begin{array}{cl}
            (1,1)&  \text { if } d_0=-1\\
            (0,0)&  \text { if } d_0=1             
            \end{array}\right.$.  
    \end{itemize}
 Note that, from the above relations, we have:
\begin{itemize}
    \item For every $0<i\leq f-1$, if $d_0=1$, $d_{i}=-1$, and $d_{j}=0$ for every $i<j\leq f-1$, 
    or $d_0=-1$, $d_{i}=1$, and $d_{j}=0$ for every $i<j\leq f-1$, $\vec{w}$ does not exist; 
    \item For every $0\leq i\leq f-1$, if $d_0=1$, $d_{i}=1$, and $d_{j}=0$ for every $i<j\leq f-1$, or $d_0=-1$, $d_{i}=-1$, and $d_{j}=0$ for every $i<j\leq f-1$, there is exactly one $\vec{w}$; 
    \item If $d_0=0$, there are exactly two choices for $\vec{w}$. 
\end{itemize}
 We write $\overline{\psi}_{\vec{w},\vec{d}}$ for such $\overline{\psi}$, and we have $\overline{\psi}_{\vec{w},\vec{d}}=\overline{\psi}_{\vec{w}',\vec{d}'}$ if and only if $(\vec{w},\vec{d})=(\vec{w}',\vec{d}').$  
\end{theorem}
\begin{proof}
    The two relations of $\vec{w}$ and $\vec{d}$ follow from Proposition \ref{19}, Table \ref{a} and Table \ref{c} in Proposition \ref{11}. \par
    For every $0<i\leq f-1$, if $d_0=1$, $d_{i}=-1$, and $d_{j}=0$ for every $i<j\leq f-1$, we deduce from the second relation that $w_{f-1}=w_0=0$, and from the first relation    that $w_i=1$ and $w_i=w_{i+1}=\cdots=w_{f-1}$, which is a contradiction; 
    if $d_0=-1$, $d_{i}=1$, and $d_{j}=0$ for every $i<j\leq f-1$, then $\vec{w}$ does not exist for similar reason.\par
    For every $0\leq i\leq f-1$, if $d_0=1$, $d_{i}=1$, and $d_{j}=0$ for every $i<j\leq f-1$, we will show the uniqueness of $\vec{w}$. By the relations above, we have $w_0=0$, $w_j=0$ for $i\leq j\leq f-1$. For $1<j<i$, we have if $d_j=-1$ or $1$ , $w_j$ is determined by $d_j$; if $d_j=0$, $w_j$ is determined by the nearest $d_k\neq 0$ $(k<j)$. 
    If $d_0=-1$, $d_{i}=-1$, and $d_{j}=0$ for every $i<j\leq f-1$, $\vec{w}$ is uniquely determined by $\vec{d}$ for similar reason.\par 
    Next for each choice of $\vec{d}$ in the second case, to show $\overline{\psi}_{\vec{w},\vec{d}}\in W_D(\overline\rho)$, by Lemma \ref{P}, it suffices to find $\vec{u},\vec{v}\in\{0,1\}^{\ZZ/f\ZZ}$ such that $\overline{\psi}_{\vec{w},\vec{d}}=\overline{\psi}_{\vec{u},\vec{v}}$. Let $v_{f-1}=0$, $v_{k-1}\equiv d_k\ (\op{mod}2)$ for $0\leq k\leq f-1$, and let $u_0=v_0$, $u_k=0$ for $i\leq k\leq f-1$; for $0<k<i$, 
    $u_k$ is uniquely determined in the sense that, from Table \ref{a}, if $d_k=-1$ or $1$, $u_k$ is determined by $(v_{k-1},v_k,d_k)$, and if $d_k=0$, $u_k$ is determined by $(v_{k-1},v_k,u_{k-1})$. For every $0\leq i\leq f-1$, if $d_0=-1$, $d_{i}=-1$, and $d_{j}=0$ for every $i<j\leq f-1$, let $v_{f-1}=0$, $v_{k-1}\equiv d_k\ (\op{mod}2)$ for $0\leq k\leq f-1$, 
    and let $u_0\equiv 1+v_0\ (\op{mod}2)$, $u_k=1$ for $i\leq k\leq f-1$; for $0<k<i$, $u_k$ is uniquely determined for similar reason. \par
    If $d_0=0$, then $w_0$ can be $0$ or $1$, and other $w_i$'s are uniquely determined by $w_0$ and the two relations above. Let $v_{f-1}=1$, and $v_{k-1}\equiv d_k\ (\op{mod}2)$ otherwise, then let $u_0\equiv w_0-v_0\ (\op{mod}2)$, and other $u_i$'s are uniquely determined. \par
    If $\overline{\psi}_{\vec{w},\vec{d}}=\overline{\psi}_{\vec{w}',\vec{d}'}$, then 
    \begin{equation*}
        \begin{split}
            0&\equiv(\sum_{i=0}^{f-1} q^{w_{i}} p^{i} r_{i}+(1-q)\sum_{i=0}^{f-1}d_ip^i)-(\sum_{i=0}^{f-1} q^{w'_{i}} p^{i} r_{i}+(1-q)\sum_{i=0}^{f-1}d'_ip^i)\ (\op{mod}q^2-1)\\
             &\equiv(q-1)\sum_{i=0}^{f-1}p^i((w_i-w'_{i})r_i-(d_i-d'_{i}))\ (\op{mod}q^2-1).
        \end{split}
    \end{equation*}
    Note that $1\leq r_0\leq p-2$ and $0\leq r_i\leq p-3$ for $i>0$ since $\overline{\rho}$ is generic. Since 
    $\sum_{i=0}^{f-1}p^i\left|(w_i-w'_{i})r_i-(d_i-d'_{i})\right|\leq \sum_{i=0}^{f-1}p^i(r_i+2)< q+1$, we have $\sum_{i=0}^{f-1}p^i((w_i-w'_{i})r_i-(d_i-d'_{i}))=0$. 
    Since $\left|(w_0-w'_{0})r_0-(d_0-d'_{0})\right|\leq p$ and $\left|(w_i-w'_{i})r_i-(d_i-d'_{i})\right|\leq p-1$ for $i>0$, we have $(w_i-w'_{i})r_i-(d_i-d'_{i})=0$ for all $i$, or for some $0<i\leq f-1$, $(w_0-w'_{0})r_0-(d_0-d'_{0})=p$, 
    $(w_i-w'_{i})r_i-(d_i-d'_{i})=-1$, $(w_j-w'_{j})r_i-(d_j-d'_{j})=p-1$ for $1<j<i$, and $(w_j-w'_{j})r_i-(d_j-d'_{j})=0$ for $i<j\leq f-1$. 
    For the first case, since $w_i\leq w'_i$ if $d_i>d'_i$ by relations above, we have $d_i=d'_i$ and $(w_i-w'_{i})r_i=0$ for all $i$. 
    Since $1\leq r_0\leq p-2$, we have $w_0=w'_0$. Then $\vec{w}$ is uniquely determined by $(b_i)_i$ once $w_0$ is fixed. For the second case, we have $w_0=1$, $w'_0=0$, $d_0=-1$, $d_0-1$. Then $w_{f-1}=1$, $w'_{f-1}=0$ by the second relation which implies $d_{f-1}\in\{-1,0\}$, $d'_{f-1}\in\{0,1\}$ by the first relation. 
    Since $(w_{f-1}-w'_{f-1})r_{f-1}-(d_{f-1}-d'_{{f-1}})=-1$ or $0$, we have this should be $0$ and $d_{f-1}=d'_{{f-1}}=0$. Then we have $w_{f-2}=1$, $w'_{f-2}=0$ by the first relation which implies $d_{f-2}\in\{-1,0\}$, $d'_{f-2}\in\{0,1\}$. We repeat this procedure and deduce that $(w_i-w'_{i})r_i-(d_i-d'_{i})=0$ for all $i\leq f-1$ which is a contradiction. \par
    
\end{proof}

\bigskip

\noindent  Department of Mathematical Sciences, Tsinghua University, Beijing, 100084\\
{\it E-mail:} {\ttfamily yang-che20@mails.tsinghua.edu.cn}\\

\noindent  Academy for Multidisciplinary Studies, Capital Normal University, Beijing, 100048\\
{\it E-mail:} {\ttfamily haoran@cnu.edu.cn}\\

\end{document}